\documentclass[11pt,reqno]{amsart}

\usepackage[a4paper, left=30mm,right=30mm,top=30mm,bottom=30mm,marginpar=25mm]
{geometry}
\usepackage[english]{babel}
\usepackage[utf8x]{inputenc}
\usepackage{setspace}
\usepackage{enumitem} 
\usepackage{amsmath}
\usepackage{mathtools}
\usepackage[normalem]{ulem}
\usepackage{amsfonts}
\usepackage{amssymb}
\usepackage{fancyhdr}
\usepackage{textcomp}
\usepackage{graphicx} 
\usepackage{float}
\usepackage{listings}
\usepackage{amssymb}
\usepackage{color}

\newtheorem{theorem}{Theorem}[section]
\newtheorem{lemma}{Lemma}[section]

\newtheorem{definition}{Definition}[section]
\newtheorem{remark}{Remark}[section]
\newcommand{\ds}{\displaystyle}
\newcommand{\ra}{\rangle}
\newcommand{\la}{\langle}
\newcommand{\lf}{\lambda_{F}}

\newcommand{\lv}{\lambda_{v}}
\newcommand{\lvi}{\lambda_{v_i}}
\newcommand{\lt}{\lambda_{\theta}}
\newcommand{\fe}{F^\varepsilon}
\newcommand{\ve}{v^\varepsilon}

\newcommand{\te}{\theta^\varepsilon}

\newcommand{\bn}{\boldsymbol{\nu}}
\newcommand{\bg}{\boldsymbol{\gamma}}
\newcommand{\bs}{\boldsymbol{\sigma}}

\newcommand{\ph}{\frac{\partial\Phi^B}{\partial F_{i\alpha}}(F)}
\newcommand{\phl}{\frac{\partial\Phi^B}{\partial F_{i\alpha}}(\lambda_F)}
\newcommand{\phb}{\frac{\partial\Phi^B}{\partial F_{i\alpha}}(\bar{F})}
\newcommand{\phv}{\frac{\partial\Phi^B}{\partial F_{i\alpha}}(F)v_i}

\newcommand{\phbv}{\frac{\partial\Phi^B}{\partial F_{i\alpha}}(\bar{F})\bar{v}_i}

\definecolor{listinggray}{gray}{0.9}
\definecolor{lbcolor}{rgb}{0.9,0.9,0.9}

\allowdisplaybreaks
\def\del{\partial}

\begin{document}
	
	\numberwithin{equation}{section}
	
	\title[Measure-valued weak versus strong uniqueness]{Measure-valued solutions for the equations of polyconvex adiabatic thermoelasticity}

	\author[C. Christoforou]{Cleopatra Christoforou}
	\address[Cleopatra Christoforou]{Department of Mathematics and Statistics,
		University of Cyprus, Nicosia 1678, Cyprus.}
	\email{christoforou.cleopatra@ucy.ac.cy} 
	\author[M. Galanopoulou]{Myrto Galanopoulou}
	\address[Myrto Galanopoulou]{Computer, Electrical, Mathematical Sciences \& Engineering Division, King Abdullah University of Science and Technology (KAUST), Thuwal, Saudi Arabia.}
	\email{myrtomaria.galanopoulou@kaust.edu.sa}
	\author[A. E. Tzavaras]{Athanasios E. Tzavaras}
	\address[Athanasios E. Tzavaras]{Computer, Electrical, Mathematical Sciences \& Engineering Division, King Abdullah University of Science and Technology (KAUST), Thuwal, Saudi Arabia. }
%	\address[]{\and  Institute of Applied and Computational Mathematics, FORTH, Heraklion, Greece.}
	\email{athanasios.tzavaras@kaust.edu.sa}
	
	\date{\today}
	
	\keywords{Polyconvex; Thermoelasticity; Uniqueness; Measure-valued solutions; Relative entropy}
	\subjclass[2010]{35Q74 (primary), and 35A02, 35L65, 74B20, 74D10, 80A17, 74H25 (secondary)} 
	
	\begin{abstract}
		For the system of polyconvex adiabatic thermoelasticity, we define a notion of dissipative measure-valued solution,
		which can be considered as the limit of a viscosity approximation. We embed the system into a symmetrizable hyperbolic one in order to derive the relative
		entropy. However, we base our analysis in the original variables, instead of the 
		symmetric ones (in which the entropy is convex) and we prove measure-valued weak versus strong uniqueness using the 
		averaged relative entropy inequality. 
	\end{abstract}

	\maketitle 
	%\tableofcontents

	\section{Introduction}
	For systems of hyperbolic conservation laws, the class of measure-valued solutions \cite{dP85} provides a notion of solvability 
	vast enough to support a global existence theory. These solutions usually arise as limits of converging sequences satisfying an
	approximating parabolic problem \cite{MR3468916}. As these solutions are considered to be very weak, it is crucial to
	examine their stability properties with respect to classical solutions and to attempt that in their natural energy framework. 
	The relative entropy method of Dafermos \cite{doi:10.1080/01495737908962394,MR546634} and DiPerna \cite{MR523630} provides an
	 analytical framework upon which one can examine such questions, and has been tested in a variety of contexts  (e.g. \cite{bds11,MR1831175,GSW2015,christoforou2016relative,fnkt17,bf18}).
%	 It is a way of measuring how close two solutions of a given physical system are, as they evolve in time, and consequentially
%	determine whether they share common properties. 
	
	In this article, we derive (in the appendix) a framework of dissipative measure valued solutions for the system of adiabatic polyconvex thermoelasticity,
	motivated by approximating that system by the system of thermoviscoelasticity on the natural energy framework.  
	The relative entropy method is then used to show weak-strong uniqueness for polyconvex thermoelasticity in the class of measure-valued solutions. 
	The main novelty of this work is the derivation of the averaged relative entropy inequality with respect to a dissipative
	measure-valued solution. This solution is defined by means of generalized Young measures, describing both oscillatory
	and concentration effects. The analysis is based on the embedding  of polyconvex thermoelasticity into an augmented, symmetrizable, 
	hyperbolic system,  \cite{CGT2017}. However, the embedding cannot be used in a direct manner, 
	and notably,  instead of working with the extended variables, we base our analysis on the parent system in the original 
	variables using the weak stability properties of some transport-stretching identities,
	which allow us to carry out the calculations by placing minimal regularity assumptions in the energy framework.
	
	Consider the system of adiabatic thermoelasticity,
	\begin{align}
	\begin{split}
	\label{sys1adiab.1}
	\partial_t F_{i\alpha}&=\partial_{\alpha}v_i \\ 
	\partial_t v_i&=\partial_{\alpha}\Sigma_{i\alpha}\\
	\partial_t\left(\frac{1}{2}|v|^2+e\right)&=\partial_{\alpha}(\Sigma_{i\alpha}v_i)+r
	\end{split}
	\end{align}
	describing the evolution of a thermomechanical process $\big ( y(x,t) , \theta(x,t) \big) \in \mathbb{R}^3\times\mathbb{R}^+$ with $(x,t)\in\mathbb{R}^3\times\mathbb{R}^+$.
	Here, $F \in \mathbb{M}^{3\times3}$ stands for the deformation gradient, $F = \nabla y$, while $v = \partial_t y$ is 
	the velocity of the motion $y$ and $\theta$ is the temperature. The condition 
	\begin{equation}
	\label{constraint}
	\partial_{\alpha}F_{i\beta}=\partial_{\beta}F_{i\alpha}, \qquad i,\alpha,\beta=1,2,3 \, ,
	\end{equation}
	imposes that $F$ is a gradient and comes from equation \eqref{sys1adiab.1}$_1$ as an involution inherited from the initial data. The stress is denoted as $\Sigma_{i \alpha }$, the internal energy as $e$ and the radiative heat supply as 
	$r$. The requirement of consistency with the Clausius-Duhem 
	inequality imposes that the elastic stresses $\Sigma$, the entropy $\eta$ and the internal energy $e$  are related to
	the free-energy function $\psi$ via the constitutive theory
	\begin{align}
	\label{con.rel.aug}
	\psi=\psi(F,\theta),\quad
	\Sigma=\frac{\partial\psi}{\partial F},\quad
	\eta=-\frac{\partial\psi}{\partial\theta},\quad
	e=\psi+\theta\eta \, .
	\end{align}
	
	Here, we work in the polyconvex regime where the free energy $\psi$ factorizes as a uniformly convex function of the 
	null-Lagrangian vector $\Phi(F)$ and the temperature $\theta,$ namely
	\begin{equation}
	\label{cond.poly}
	\psi (F, \theta) =  \hat{\psi}(\Phi(F),\theta)\,,
	\end{equation}
	satisfying
	\begin{equation}
	\label{cond.redgibbs}
	\hat\psi_{\xi \xi} (\xi ,  \theta) > 0 \, , \quad \hat\psi_{\theta \theta} (\xi, \theta) < 0 \, ,
	\end{equation}
	where $\hat{\psi}=\hat{\psi}(\xi,\theta)$ is a strictly convex function on $\mathbb{R}^{19}\times \mathbb{R}^+$ and
	$\Phi(F)=(F,\mathrm{cof} F,\det F)\in\mathbb{M}^{3\times3}\times\mathbb{M}^{3\times3}\times\mathbb{R}.$
	
	The main result of this article is the weak-strong uniqueness of polyconvex adiabatic thermoelasticity (\ref{sys1adiab.1})-(\ref{cond.poly})
	in the class of dissipative measure-valued solutions. The advantage of the dissipative framework is
	that the averaged energy equation holds in its integrated form.
	Even though this notion of solutions is generally considered to be very weak, not possessing detailed information, 
	this result contributes to a long list of similar works
	\cite{ab97,dm87,MR1831179,MR1831175,christoforou2016relative} on hyperbolic systems of conservation laws,
	pointing out the importance of this framework in the analysis of such physical problems. Unlike the case of scalar
	conservation laws \cite{dP85,MR725524}, where the theory of Young measures suffices to deal with
	nonlinearities and overcome oscillatory behaviors, when it comes to hyperbolic systems, one must take into account
	the formation of both oscillations and concentrations. In our case, the concentration effects are described through
	a concentration measure, which appears in the energy equation since the Fundamental Lemma of Young measures cannot
	represent the weak limits of $\frac{1}{2}|v|^2+e,$ due to lack of $L^1$ precompactness. This is illustrated in Appendix \ref{AppA}.
	Thus we turn our attention to the theory of generalized Young
	measures \cite{ab97,dm87,bds11,MR1831175,christoforou2016relative,GSW2015}
	and apply the relative entropy formulation to compare a dissipative measure-valued solution to polyconvex
	thermoelasticity against a strong solution.
	
	We organize this paper as follows: In Section \ref{sec2}, we define the notion of dissipative measure-valued solutions
	for polyconvex thermoelasticity. This definition comes as a result of the limiting process we discuss in Appendix 
	\ref{AppA}, starting from the associated viscous problem. Section \ref{sec3} is dedicated to the study of the generated 
	Young measure and the concentration measure, which is a well-defined, nonnegative Radon measure for a subsequence of 
	approximate solutions coming from a uniform bound on the energy. In Section \ref{sec4} we calculate the averaged 
	relative entropy inequality (\ref{rel.en.id.mv}) and in Section \ref{sec5} we use it to prove the main theorem on 
	uniqueness of strong solutions in the class of measure-valued solutions. The proof is heavily based on the estimates
	(\ref{bound4}) and (\ref{bound6})-(\ref{bound8}) on the relative entropy, namely Lemmas \ref{lemma1}, \ref{lemma2},
	which are stated and proved at the level of the original variables, instead of the extended ones. As a result, we 
	only assume quite minimal growth hypotheses on the constitutive functions, which guarantee all the necessary technical 
	requirements for the dissipative measure-valued versus strong uniqueness to hold. Additionally, the proof is carried on with 
	respect to a dissipative solution which satisfies an averaged and integrated version of the energy equation, where the 
	concentration measure appears. This setting has the strong advantage that we need no artificial integrability 
	restrictions on the energy equation. Similar results are available for the incompressible Euler equations \cite{bds11}, 
	for  polyconvex elastodynamics \cite{MR1831175}, and for the isothermal gas dynamics system \cite{GSW2015}. 
	%Finally, in section \ref{sec5}, prior to the main uniqueness theorem, we bound all the terms consisting the relative entropy inequality (\ref{rel.en.id.mv}).

	\section{Measure-valued solutions for polyconvex adiabatic thermoelasticity} \label{sec2}
	Consider the system of adiabatic thermoelasticity (\ref{sys1adiab.1}),~(\ref{constraint}) together with the entropy 
	production identity
	\begin{equation}
	\label{S2: entropyadiab}
	\partial_t \eta =\frac{r}{\theta} \, ,
	\end{equation}
	under the constitutive theory (\ref{con.rel.aug}) and the polyconvexity hypothesis \eqref{cond.poly}.
	To avoid unnecessary technicalities, henceforth we work in a domain $Q_T=\mathbb{T}^d\times[0,T)$, where $\mathbb{T}^d$ is the torus, $d=3$ and $T\in[0,\infty)$.
	 In the polyconvex case, the Euler-Lagrange equation, 
	\begin{equation}
	\label{null-Lan}
	\partial_{\alpha}\left( \frac{\del \Phi^B}{\del F_{i\alpha}} (\nabla y ) \right)=0 \, , \quad B=1,\dots,19\;,
	\end{equation}
	formulated for the vector of the minors
	$\Phi(F)=(F,\mathrm{cof} F,\det F)\in\mathbb{M}^{3\times3}\times\mathbb{M}^{3\times3}\times\mathbb{R},$
	holds for any motion $y(x,t)$ and together with the kinematic compatibility equation (\ref{sys1adiab.1})$_1$ and
	(\ref{constraint}), allows to express $\partial_t \Phi^B(F)$ as
	\begin{align}
	\label{nul-Lang1}
	\partial_t \Phi^B(F)
	%=\ph \partial_t F_{i\alpha}
	=\ph \partial_{\alpha}v_i=\partial_{\alpha}\left(\phv\right),
	\end{align}
	for any deformation gradient $F$ and velocity field $v.$ Additionally, the stress tensor $\Sigma$ becomes
	\begin{align}
	\label{Sigma}
	\Sigma_{i\alpha}=\frac{\partial\psi}{\partial F_{i\alpha}}(F,\theta)=\frac{\partial }{\partial F_{i\alpha}}\left(\hat{\psi}(\Phi(F),\theta)\right)=\frac{\partial \hat{\psi}}{\partial\xi^B}(\Phi(F),\theta)\ph.
	%\\
	%&=\frac{\partial g}{\partial F_{i\alpha}}(F,Z,w,\theta)+\frac{\partial g}{\partial Z_{k\gamma}}(F,Z,w,\theta)\frac{\partial(\mathrm{cof} F)^{k\gamma}}{\partial F_{i\alpha}}
	%+\frac{\partial g}{\partial w}(F,Z,w,\theta)\frac{\partial(\det F)}{\partial F_{i\alpha}},
	\end{align}
	The first nine components of $\Phi(F)$ are the components of $F,$ therefore (\ref{con.rel.aug}) implies that 
	we can express the entropy $\eta$ and the internal energy $e$ with respect to the null-Lagrangian vector $\Phi(F),$ namely
	\begin{align}
	\label{eta&e}
	\begin{split}
	\eta(F,\theta)&=-\frac{\partial\psi}{\partial\theta}(F,\theta)=-\frac{\partial \hat{\psi}}{\partial\theta}(\Phi(F),\theta) = :\hat{\eta}( \Phi(F) , \theta)
	%=:\hat{\eta}(\Phi(F),\theta)
	,\\
	%\label{e.tilde}
	e(F,\theta)&=\psi(F,\theta)-\theta\frac{\partial\psi}{\partial\theta}(F,\theta)=\hat{\psi}(\Phi(F),\theta)-\theta \frac{\partial\hat{\psi}}{\partial\theta}(\Phi(F),\theta)
	= \hat{e}(\Phi(F),\theta) \, ,
	\end{split}
	\end{align}
	where we have set
	\begin{equation}
	\label{def.hats}
	\hat{\eta}( \xi , \theta) := -\frac{\partial \hat{\psi}}{\partial\theta}( \xi ,\theta) \, , \quad \hat e(\xi ,\theta) := \hat \psi(\xi ,\theta)- \theta\frac{\partial \hat \psi}{\partial\theta}(\xi,\theta) \, .
	\end{equation}
	This allows to supplement the equations of polyconvex thermoelasticity (\ref{sys1adiab.1}),~(\ref{constraint})
	with  \eqref{nul-Lang1} and write
	\begin{align}
	\begin{split}
	\label{sys2}
	\partial_t \Phi^B(F)&=\partial_{\alpha}\left(\phv \right)\\ 
	\partial_t v_i&=\partial_{\alpha}\left(\frac{\partial\hat{\psi}}{\partial \xi^B}(\Phi(F),\theta) \frac{\partial\Phi^B}{\partial F_{i\alpha}}(F)\right)\\
	\partial_t \left(\frac{1}{2}|v|^2+\hat{e}(\Phi(F),\theta)\right)&=\partial_{\alpha}\left(\frac{\partial\hat{\psi}}{\partial \xi^B}(\Phi(F),\theta) \frac{\partial\Phi^B}{\partial F_{i\alpha}}(F) v_i\right)+r\\
	\partial_{\alpha}F_{i\beta}&=\partial_{\beta}F_{i\alpha}
	\end{split}
	\end{align}
	while the entropy production identity (\ref{S2: entropyadiab}) becomes
	\begin{equation}
	\label{entropy.prod}
	\partial_t \hat{\eta}(\Phi(F),\theta) =\frac{r}{\theta}.
	\end{equation}

	This implies that  $(\xi = \Phi(F) , v, \theta )$ satisfies the augmented system 
	\begin{align}
	\begin{split}
	\label{sysaug}
	\partial_t \xi^B &=\partial_{\alpha}\left(\phv \right)\\ 
	\partial_t v_i&=\partial_{\alpha}\left(\frac{\partial\hat{\psi}}{\partial \xi^B}(\xi,\theta) \frac{\partial\Phi^B}{\partial F_{i\alpha}}(F)\right)\\
	\partial_t \left(\frac{1}{2}|v|^2+\hat{e}(\xi,\theta)\right)&=\partial_{\alpha}\left(\frac{\partial\hat{\psi}}{\partial \xi^B}(\xi,\theta) \frac{\partial\Phi^B}{\partial F_{i\alpha}}(F) v_i\right)+r\\
	\partial_{\alpha}F_{i\beta}&=\partial_{\beta}F_{i\alpha}
	\end{split}
	\end{align}
	that consists of conservation laws in $\mathbb{R}^{23}$ subject to the involution \eqref{sysaug}$_4$. The augmented system satisfies the entropy production identity
	\begin{equation}
	\label{entropy.prodaug}
	\partial_t \hat{\eta}(\xi ,\theta) =\frac{r}{\theta} 
	\end{equation}
	and is thus symmetrizable; see \cite{CGT2017} for further details.

	The system \eqref{sysaug} belongs to a 
	general class of hyperbolic conservation laws of the form $$\partial_t A(U)+\partial_{\alpha}f_{\alpha}(U)=0,$$ 
	$U:\mathbb{R}^d\times\mathbb{R}^+\to\mathbb{R}^n,$ studied in~\cite{christoforou2016relative}. Due to \eqref{entropy.prodaug} it
	is symmetrizable, hyperbolic in the extended variables. A general theory including a theorem
	establishing recovery of classical solutions from dissipative measure–valued solutions for hyperbolic systems
	endowed with a convex entropy, was developed in \cite{christoforou2016relative}. We note that since in the variables
	$(F,v,\theta)$ system (\ref{sys2}) is not equipped with a convex entropy, we cannot treat this problem as a direct 
	application of the general setting developed in \cite{christoforou2016relative}. 
	In~\cite{CGT2017}, system~\eqref{sys1adiab.1}--\eqref{cond.redgibbs} was studied by augmenting it to~\eqref{sys2} using the relative entropy method
	in order to prove convergence from thermoviscoelasticity to the system \eqref{sys1adiab.1}--\eqref{cond.redgibbs}.
	The objective in the present paper is to prove a weak-strong uniqueness theorem in the context of measure-valued solutions. 
	This requires to work at the level of the original rather than the augmented system what presents various technical challenges.
		
	Following the theory on generalized Young measures \cite{ab97,christoforou2016relative,dm87}, we define a
	dissipative measure-valued solution to polyconvex thermoelasticity, which involves a parametrized Young measure
	$\bn=\bn_{(x,t)}$ describing the oscillatory behavior of the solution and a Radon measure $\bg\in\mathcal{M}^+(Q_T)$
	describing concentration effects. According to the analysis in Appendix \ref{AppA}, we can treat dissipative
	measure-valued solutions as limits of an approximating solution for the associated viscous problem,
	that satisfy an averaged and integrated energy equation. The reason behind the formation of concentrations, lies with
	the fact that the energy function $(x,t)\mapsto  |v|^2 + e(F,\theta) $ is not weakly precompact 
	in $L^1$ and thus, the Young measure representation fails. Since the only uniform bound at one's disposal is on the 
	energy, the way we construct these solutions corresponds to a minimal framework obtained from this natural bound, for 
	viscosity approximations of the adiabatic thermoelasticity system. The analysis in Appendix~\ref{AppA} leads to the following definition:
	\begin{definition} \label{def.mv.F}
		A dissipative measure valued solution to polyconvex thermoelasticity (\ref{sys2}), (\ref{entropy.prod})
		consists of a thermomechanical process $(y(t,x), \theta(t,x)) :  [0,T] \times \mathbb{T}^3 \to \mathbb{R}^3 \times \mathbb{R}^+$, 
		\begin{align}
		\label{regulatity.y} 
		y \in W^{1,\infty}(L^2(\mathbb{T}^3)) \cap L^\infty (W^{1,p} (\mathbb{T}^3)) \, , \quad \theta\in L^{\infty}(L^{\ell} (\mathbb{T}^3)   ) \,  ,
		\end{align}
%		for $p\geq 4,$ $q,\rho\geq 2,$ $\ell>1,$ for which there holds 
%
		a parametrized family of probability Young measures $\bn=\bn_{(x,t)\in\bar{Q}_T},$ 
		with averages
		\begin{align*}
		F=\left\la\bn,\lambda_F \right\ra,\quad v=\left\la\bn,\lv\right\ra, \quad \theta=\left\la\bn,\lt\right\ra \, , 
		\end{align*} 
		and a nonnegative Radon measure $\bg\in\mathcal{M}^+(Q_T)$,
		where
		\begin{align}
		\label{regulatity.Fvtheta}
		F=\nabla y\in L^{\infty}(L^p), \quad  v=\partial_t y\in L^{\infty}(L^2) \, , 
		\end{align}
		$$
		\Phi(F)=(F, \mathrm{cof} F, \det F)\in L^{\infty}(L^p)\times L^{\infty}(L^q)\times L^{\infty}(L^{\rho}) \, , 
		$$
		$p\geq 4,$ $q \geq 2$, $\rho > 1$, $\ell>1$,  which satisfy the averaged equations
		\begin{align}
		\begin{split}
		\label{mv.sol.Fvt}
		\partial_t \Phi^B(F)&=\partial_{\alpha}\left(\phv \right) \\ 
		\partial_t \left\la\bn,\lvi\right\ra
		&=\partial_{\alpha}\left\la\bn,\frac{\partial \hat{\psi}}{\partial\xi^B}(\Phi(\lambda_{F}),\lambda_{\theta})\frac{\partial\Phi^B}{\partial F_{i\alpha}}(\lf)\right\ra\\
		\partial_t \left\la\bn,\hat{\eta}(\Phi(\lambda_{F}),\lt)\right\ra &\geq \left\la\bn,\frac{r}{\lt}\right\ra
		\end{split}
		\end{align}
		in the sense of  distributions, together with the integrated form of the  averaged energy equation,
		\begin{small}
			\begin{align}
			\begin{split}
			\label{mv.sol.Fvt.energy}
			\int& \varphi(0)  \left(\left\la\bn,\frac{1}{2}|\lv|^2+\hat{e}(\Phi(\lf),\lt)\right\ra (x,0)\:dx +\bg_0(dx) \right)\\
			&+\int_0^T
			\int  \varphi^{\prime}(t) \left(\left\la\bn,\frac{1}{2}|\lv|^2+\hat{e}(\Phi(\lf),\lt) \right\ra(x,t)  \, dx\:dt +\bg(dx\:dt) \right)  \\
			&=-\int_0^T\int \left\la\bn,r\right\ra\varphi(t)\:dx\:dt,
			\end{split}
			\end{align}
		\end{small}
		for all $\varphi \in C^1_c[0,T].$
	\end{definition}

	In this definition, the first equation (\ref{mv.sol.Fvt})$_1$ holds in a classical weak sense under the regularity conditions (\ref{regulatity.Fvtheta}),(\ref{regulatity.y}) 
	placed on the motion and 
	its derivatives for $p\geq 4,$ $q \geq  2,$ $\rho, \ell>1,$   as a consequence of the weak continuity of the null-Lagrangian vector $(F,\mathrm{cof} F,\det F)$ and the weak continuity of the transport-stretching  identities 
	\begin{equation}
	\label{transport.stretching}
	\begin{aligned}
	\partial_t F_{i\alpha}&=\partial_{\alpha}v_i  \\
%     \partial_t (\mathrm{det}F)&=\partial_{\alpha}\left(\frac{\partial \mathrm{det}F}{\del F_{i\alpha}} v_i\right)
        \partial_t \det F&=\partial_\alpha \bigl((\mathrm{cof} F)_{i\alpha}v_i\bigr)
	\\
%      \partial_t (\mathrm{cof}F)_{\kappa\gamma}&=\partial_{\alpha}\left(\frac{\partial (\mathrm{cof}F)^{\kappa\gamma}}{\del F_{i\alpha}}  v_i\right) 
         \partial_t(\mathrm{cof} F)_{k \gamma}&=\partial_\alpha(\epsilon_{ijk}\epsilon_{\alpha\beta\gamma}F_{j\beta}v_i).
	\end{aligned}
	\end{equation}
	We summarize the corresponding results, taken out of  \cite{MR0475169} and \cite{MR1831179}, in the following lemma.
	As the weak continuity property is important for the forthcoming analysis, we present the proof here for the reader's convenience.

	\begin{lemma} \label{lemma.weak.minors}\cite[Lemma 6.1]{MR0475169}, \cite[Lemmas 4,5]{MR1831179}
	The followings hold true:
	\begin{itemize}
	\item[$(i)$] For $y \in W^{1,\infty}(L^2(\mathbb{T}^3)) \cap L^\infty (W^{1,p} (\mathbb{T}^3))$ with $p \ge 4$,
	$F= \nabla y$ and $v = \del_t y$,  the formulas \eqref{transport.stretching} hold in the sense of distributions.
	\item[{$(ii)$}] Suppose the family $\{y^\varepsilon \}_{\varepsilon > 0}$, where  $y^{\varepsilon} : [0,\infty) \times \mathbb{T}^3 \to \mathbb{R}^3$ satisfies
		\begin{align}
		\label{y.reg}
		y^{\varepsilon} \; \mbox{is unifomly bounded in } \; W^{1,\infty}(L^2(\mathbb{T}^3)) \cap L^\infty (W^{1,p} (\mathbb{T}^3)) \, ,
		\end{align}
		and let $v^\varepsilon = \del_t y^\varepsilon$, $F^\varepsilon = \nabla y^\varepsilon$. Then, along a subsequence,
		\begin{align*}
		(F^{\varepsilon},\mathrm{cof}F^{\varepsilon},\det F^{\varepsilon})\rightharpoonup(F, \mathrm{cof}F,\det F)  ,\:\text{weakly in\:} 
		L^{\infty}(L^p) \times  L^{\infty}(L^q)\times L^{\infty}(L^{\rho})
		\end{align*}
		with $p\ge 2,\:q\ge\frac{p}{p-1},\:q\ge\frac{4}{3}$, $\rho > 1$. Moreover,  if $p\ge 4$ the identities \eqref{transport.stretching} 
		are weakly stable in the regularity class \eqref{y.reg}.
		\end{itemize}
	\end{lemma}
	
	\begin{proof}
		We note the formulas, for smooth maps,
		\begin{align}
		\label{cof.det.1}
		\begin{split}
		(\mathrm{cof}F)_{i\alpha}&=\frac{1}{2}\epsilon_{ijk}\epsilon_{\alpha\beta\gamma}F_{j\beta}F_{k\gamma},\\
		\det F&=\frac{1}{6}\epsilon_{ijk}\epsilon_{\alpha\beta\gamma}F_{i\alpha}F_{j\beta}F_{k\gamma}
		=\frac{1}{3}(\mathrm{cof} F)_{i \alpha}F_{i \alpha}
		\end{split}
		\end{align}
                 and 
		\begin{align*}
		\partial_t \det F&=\partial_\alpha \bigl((\mathrm{cof} F)_{i\alpha}v_i\bigr)\\
		\partial_t(\mathrm{cof} F)_{k \gamma}&=\partial_\alpha(\epsilon_{ijk}\epsilon_{\alpha\beta\gamma}F_{j\beta}v_i).
		\end{align*}

		\textbf{Step $1.$} For $y \in W^{1,\infty}(L^2(\mathbb{T}^3)) \cap L^\infty (W^{1,p} (\mathbb{T}^3))$, we extend $y$ to a function defined for all times, 
		by putting $y(t,x)=y(0,x),$ for $t\leq 0.$ The extended $y$ belongs to the same regularity class. 
		Define the convolution (in space and time) $y_{\epsilon}:=y\star f_{\epsilon},$ where 
		$f_{\epsilon}=\varrho_{\epsilon}(t)\prod_{i=1}^{3}\varrho_{\epsilon}(x_i),$ $\varrho_{\epsilon}=\frac{1}{\epsilon}\varrho\left(\frac{s}{\epsilon}\right),$ 
		for $\varrho\in C_0^{\infty}(\mathbb{R})$
		positive, $\int \rho(s) ds = 1$. Then $y_{\epsilon}\in C^{\infty}( \mathbb{R} \times\mathbb{T}^3)$ and such that
		for all $s<\infty,\:T>0:$ 
		$$\|\partial_ty_{\epsilon}-\partial_ty\|_{L^s([-T,T];L^2)}+\|y_{\epsilon}-y\|_{L^{\infty}([-T,T];W^{1,p})}\to 0.$$
		Let $F_{\epsilon}=\nabla y_{\epsilon}$ and $v_{\epsilon}=\partial_t y_{\epsilon}.$ Since the cofactor matrix is 
		bilinear in the components of $F,$ and the determinant is trilinear, it follows by repeated use
		of Hölder inequalities that for some numerical constant $C$,
		\begin{align*}
		\big \| \mathrm{cof}F_{\epsilon} - \mathrm{cof}F \big \|_{L^s ( L^{p/2} )}
		&\le  C \big \| F_{\epsilon} - F \big \|_{L^{2s}  ( L^{p} )}  \, \big \| |F_{\epsilon}| + |F | \big \|_{L^{2s}  ( L^{p} )}
		\\
		\big \| \det F_{\epsilon} - \det F \big \|_{L^s ( L^{p/3} )}
		&\le  C  \big \| F_{\epsilon} - F \big \|_{L^{3s}  ( L^{p} )}  \,  \left ( \big \| |F_{\epsilon}| + |F | \big \|_{L^{3s}  ( L^{p} )} \right )^2.
		\end{align*}
		We thus conclude: 
		\begin{align*}
		\mathrm{cof}F_{\epsilon}\to\mathrm{cof}F \quad \text{in\:\:} L^s(L^{p/2}) \, , \quad \det F_{\epsilon}\to \det F \quad \text{in\:\:} L^s(L^{p/3}) \, .
		\end{align*}
		Passing to the limit $\epsilon\to 0$, for $p\geq 4$, in the formulas
		\begin{align*}
		\partial_t(\mathrm{cof}F_{\epsilon})_{k\gamma}
		& =\partial_{\alpha} \big ( \epsilon_{ijk}\epsilon_{\alpha\beta\gamma} (F_{\epsilon})_{j\beta}(v_{\epsilon})_i \big ) 
		\\
		\partial_t(\det F_{\epsilon})
		&=\partial_{\alpha}((\mathrm{cof}F_{\epsilon})_{i \alpha}(v_{\epsilon})_i)
		\end{align*}
		we obtain \eqref{transport.stretching} in the sense of distributions and complete the proof of (i).
		
		\textbf{Step $2.$} Let $\{ y^\varepsilon \}_{\varepsilon > 0}$ be a family satisfying the uniform bound \eqref{y.reg} and let 
		$\fe=\nabla y^{\varepsilon}$ and $\ve=\partial_t y^{\varepsilon}.$ We adapt the proof of
		\cite[Lemma 6.1]{MR0475169}  suggesting to write the cofactor
		and the determinant in divergence form:
		\begin{align*}
		(\mathrm{cof}\fe)_{i\alpha}
		&=\frac{1}{2}\epsilon_{ijk}\epsilon_{\alpha\beta\gamma}\fe_{j\beta}\fe_{k\gamma}
		=\frac{1}{2}\partial_{\beta}(\epsilon_{ijk}\epsilon_{\alpha\beta\gamma}y^{\varepsilon}_{j}\fe_{k\gamma}),\\
		\det\fe&=\frac{1}{3}(\mathrm{cof}\fe)_{i \alpha}\fe_{i \alpha}
		=\frac{1}{3}\partial_{\alpha}(y^{\varepsilon}_{i}(\mathrm{cof}\fe)_{i \alpha}).
		\end{align*}
		With $p\geq2,\:q\ge\frac{p}{p-1},$ hypothesis (\ref{y.reg}) implies $y^{\varepsilon} \rightharpoonup y$  weakly in 
		$W^{1,2}_{loc}([0,\infty)\times\mathbb{T}^3)$ along subsequences and Rellich's theorem (for dimension $3+1$)
		implies  $y^{\varepsilon} \to y$ strongly in $L^z_{loc}([0,\infty)\times\mathbb{T}^3)$
		for $z<4.$ Additionally, $\fe\rightharpoonup F$ weak-$\ast$ in $L^{\infty}(L^p),$ for $p\geq2>\frac{4}{3}$
		the dual exponent to $4.$ Therefore, we can pass to the limit in the sense of distributions:
		\begin{align}
		\label{limit.cof}
		\frac{1}{2}\partial_{\beta}(\epsilon_{ijk}\epsilon_{\alpha\beta\gamma}y^{\varepsilon}_{j}\fe_{k\gamma})
		\rightharpoonup\frac{1}{2}\partial_{\beta}(\epsilon_{ijk}\epsilon_{\alpha\beta\gamma}y_{j}F_{k\gamma})
		=(\mathrm{cof}F)_{i\alpha}
		\end{align}
		and similarly for the determinant
		\begin{align}
		\label{limit.det}
		\frac{1}{3}\partial_{\alpha}(y^{\varepsilon}_{i}(\mathrm{cof}\fe)_{i \alpha})
		\rightharpoonup\frac{1}{3}\partial_{\alpha}(y_{i}(\mathrm{cof}F)_{i \alpha})=\det F,
		\end{align}
		since $\mathrm{cof}\fe\rightharpoonup\mathrm{cof}F$ weak-$\ast$ in $L^{\infty}(L^q)$, for $q>\frac{4}{3}$ the dual exponent to $4.$ 
		The distributional limits in (\ref{limit.cof})  and (\ref{limit.det}) coincide with the limits
		in the  weak-$\ast$ topology. Altogether we have 
		\begin{align*}
		(\mathrm{cof}\fe)_{i\alpha}&\rightharpoonup(\mathrm{cof}F)_{i\alpha},\quad\text{weak-$\ast$ in\:} L^{\infty}(L^q),
		\quad\text{for\:} q\geq\frac{p}{p-1},\:q>\frac{4}{3}\\
		\det\fe&\rightharpoonup\det F,\quad\text{weak-$\ast$ in\:} L^{\infty}(L^{\rho}).
		\end{align*}
		Next, note that $y^\varepsilon$, $F^\varepsilon$, $v^\varepsilon$ satisfy
		\begin{align*}
		\del_t (\mathrm{cof}\fe)_{k \gamma}
		&
		= \del_t \partial_{\alpha}( \frac{1}{2}\epsilon_{ijk}\epsilon_{\alpha\beta\gamma}y^{\varepsilon}_{i}\fe_{j \beta}) = 
		\partial_{\alpha}(  \epsilon_{ijk}\epsilon_{\alpha\beta\gamma}\fe_{j \beta}  v^{\varepsilon}_{i} )  , \\
		\del_t \det\fe &= \del_t \del_\alpha \big ( \frac{1}{3}  y^{\varepsilon}_{i}(\mathrm{cof}\fe)_{i \alpha} \big )
		= \del_\alpha \big ( \mathrm{cof} F^\varepsilon_{i \alpha} v_i^\varepsilon  \big ) .
		\end{align*}
		Using the weak continuity properties of $\mathrm{cof} F$ and $\det F$ and that for $p \ge 4$ equations \eqref{transport.stretching}
		hold for functions $y$ of class \eqref{y.reg}, we conclude that equations~\eqref{transport.stretching} are weakly stable.
\end{proof}

	\begin{remark}\rm
		On the definition of the dissipative measure-valued solution:
		\begin{enumerate}
			\item  Combining the requirements of Lemma \ref{lemma.weak.minors}, with those of Lemmas \ref{lemma1}
			and \ref{lemma2}, we must assume the exponents $p\geq 4,$ $q \geq 2,$ $\rho, \ell>1$.
			\item  Henceforth, we  assume the measure $\bg_0=0,$ meaning that we consider initial data with no 
			concentrations at time $t=0.$ 
			\item  Next, we  highlight why we choose to work with the system in the physical variables $(\Phi(F),v,\theta)$ instead of the extended ones 
			$(\xi,v,\theta)$:
			This allows to avoid imposing restrictive growth conditions on the constitutive functions with respect to the cofactor and the determinant
			derivatives. From previous works in isothermal polyconvex elastodynamics (e.g. \cite{MR1831175}) or even in 
			\cite{CGT2017}, it becomes evident that when considering the extended system, one has to impose growth 
			condition on terms 
			$$
			\frac{\partial\hat{\psi}}{\partial F}(\xi,\theta) \, ,  \quad 
			\frac{\partial\hat{\psi}}{\partial \zeta }(\xi,\theta) \frac{\del (\mathrm{cof}F )  }{\del F} \, , \quad 
			\frac{\partial\hat{\psi}}{\partial w }(\xi ,\theta) \frac{\del (\det F )  }{\del F}
			$$
			where $\xi = (F, \zeta, w)$, 
			in order to achieve representation of the associated weak limits via Young measures.
			The resulting regularity class of functions is far too restrictive and in particular functions with general power-like behavior 
			do not satisfy such assumptions and their weak-limits cannot be represented. 
			By contrast, if one works with the original variables, the growth hypotheses (\ref{gr.con.1})--\eqref{gr.con.2} 
			placed on $\psi(F,\theta)$ and $e(F,\theta)$, which are compatible with the constitutive theory, are also sufficient to allow
		        representation of the corresponding weak limits.
			\item The reasoning behind studying the integrated form of the averaged energy equation lies in the technical advantage that, one does not need to place any integrability condition on the right hand-side of the
			energy equation (\ref{sys2})$_3$, namely on the term 
			$$\frac{\partial \hat{\psi}}{\partial\xi^B}(\Phi(F),\theta) \, \ph v_i \, , $$ 
			since it appears as a divergence and its contribution integrates to zero.
		\end{enumerate}
	\end{remark}

	\section{Young measures and concentration measures} \label{sec3}
	%Given an internal energy function $e(F,\theta),$ the constitutive theory (\ref{con.rel.aug}) implies that one can 
	%determine the entropy $\eta$, the free energy $\psi$ and the elastic stress $\Sigma$ in the following way:
	%\begin{align*}
	%\eta(F,\theta)&=\int_{1}^{\theta} \frac{1}{s}\frac{\partial e}{\partial\theta}(F,s)\:ds,\\
	%\psi(F,\theta)&=e(F,\theta)-\theta\int_{1}^{\theta} \frac{1}{s} \frac{\partial e}{\partial F}(F,s)\:ds,\\
	%\Sigma(F,\theta)&=\frac{\partial e}{\partial F}(F,\theta)-\theta\int_{1}^{\theta} \frac{1}{s} \frac{\partial^2 e}{\partial F\partial\theta}(F,s)\:ds
	%=\theta\frac{\partial e}{\partial F}(F,1)-\theta\int_{1}^{\theta} \frac{1}{s^2} \frac{\partial e}{\partial F}(F,s)\:ds.
	%\end{align*}
	%\textcolor{red}{Therefore, if we assume the growth hypothesis
	We assume the following growth conditions on the constitutive functions $e(F,\theta),$ $\psi(F,\theta),$
	$\eta(F,\theta)$ and $\Sigma(F,\theta):$
	\begin{equation}
	\label{gr.con.1}
	c(|F|^p+\theta^{\ell})-c\leq e(F,\theta)\leq c(|F|^p+\theta^{\ell})+c\;,
	\end{equation}
	\begin{equation}
	\label{gr.con.2}
	c(|F|^p+\theta^{\ell})-c\leq\psi(F,\theta)\leq c(|F|^p+\theta^{\ell})+c\;,
	\end{equation}
	\begin{equation}
	\label{gr.con.3}
	\lim_{|F|^p+\theta^{\ell}\to\infty}\frac{|\partial_{\theta}\psi(F,\theta)|}{|F|^p+\theta^{\ell}}=
	\lim_{|F|^p+\theta^{\ell}\to\infty}\frac{|\eta(F,\theta)|}{|F|^p+\theta^{\ell}}=0\;,
	\end{equation}
	and
	\begin{equation}
	\label{gr.con.4}
	\lim_{|F|^p+\theta^{\ell}\to\infty}\frac{|\partial_{F}\psi(F,\theta)|}{|F|^p+\theta^{\ell}}=
	\lim_{|F|^p+\theta^{\ell}\to\infty}\frac{|\Sigma(F,\theta)|}{|F|^p+\theta^{\ell}}=0
	\end{equation}
	-which are consistent with the constitutive theory (\ref{con.rel.aug})-
	for some constant $c>0$ and $p\geq4,\ell>1.$ As presented in the appendix, we consider measure-valued solutions as
	limits of approximations that satisfy the uniform bound
	\begin{align}
	\label{energy_unif_bound}
	\int_{\mathbb{T}^d} \hat{e}(\Phi(F^\varepsilon),\theta^\varepsilon)+\frac{1}{2}|v^\varepsilon|^2\:dx
	\overset{(\ref{eta&e})}{=}
	\int_{\mathbb{T}^d} e(F^\varepsilon,\theta^\varepsilon)+\frac{1}{2}|v^\varepsilon|^2\:dx<C,
	\end{align}
	coming from the energy conservation equation (\ref{sys2})$_3,$ given that the radiative heat supply $r$ is a
	bounded function in $L^1(Q_T)$. The growth condition (\ref{gr.con.1}) in combination with (\ref{energy_unif_bound}) 
	suggests that the functions $\fe\in L^p,$ $\ve\in L^2,$ $\te\in L^\ell$ are all (uniformly) bounded in the respective spaces.
	The approximating sequence $U^{\varepsilon}=(\fe,\ve,\te),$ represents weak limits of the form
	\begin{align}
	\label{wk.lim.n}
	\text{wk-}\lim_{\varepsilon\to 0} f(\fe,\ve,\te)=\left\la\bn,f(F,v,\theta)\right\ra
	\end{align}
	for all continuous functions $f=f(\lf,\lv,\lt)$ such that
	\begin{equation*}
	\lim_{|\lf|^p+|\lv|^2+\lt^{\ell}\to\infty}\frac{|f(\lf,\lv,\lt)|}{|\lf|^p+|\lv|^2+\lt^{\ell}}=0\;,
	\end{equation*}
	where $\lf\in\mathbb{M}^{3\times 3}$,
	$\lv\in\mathbb{R}^3$,$\lt\in\mathbb{R}^+.$ The generated Young measure $\bn_{(x,t)}$ is associated with the motion
	$y:Q_T\to\mathbb{R}^3$ through $F$ and $v,$ by imposing that a.e.
	\begin{align*}
	F=\left\la\bn,\lf\right\ra,\quad v=\left\la\bn,\lv\right\ra, \quad \theta=\left\la\bn,\lt\right\ra
	\end{align*}
	%while
	%\begin{align*}
	%\Phi(F)=\left\la\bn,\Phi(\lf)\right\ra=\Phi(\left\la\bn,\lf\right\ra)
	%\end{align*}
	%-as a consequence of the weak continuity of the minors of $F$-
	and its action is well-defined for all functions $f$ that grow slower than the energy.
	To take into account the formation of concentration effects, we introduce the concentration measure $\bg,$ depending on the total energy. This is a well-defined nonnegative Radon measure
	for a subsequence of $\ds e(F^\varepsilon,\theta^\varepsilon)+\frac{1}{2}|v^\varepsilon|^2.$ 
	To prove this claim, let us define the sets
	\begin{align*}
	&\mathcal{F}_0=\left\{h\in C^b(\mathbb{R}^d):\;h^{\infty}(z)=\lim_{s\to\infty}h(sz)\;
	\text{exists and is continuous on}\; S^{d-1} \right\}\\
	&\mathcal{F}_1=\left\{g\in C(\mathbb{R}^d):\;g(z)=h(z)(1+|z|),\; h\in\mathcal{F}_0 \right\}.
	\end{align*}
	Let $X$ be a locally compact Hausdorff space, where we define the set of all Radon measures $\mathcal{M}(X)$ and all
	positive Radon measures $\mathcal{M}^+(X),$ while $\mathrm{Prob}(X)$ denotes all probability measures on $X.$
	Let $\Omega$ be any open subset of $\mathbb{R}^d$ and fix a Radon measure $\lambda$ on $\Omega$. We denote by
	$\mathcal{P}(\lambda;X)=L^{\infty}_w(d\lambda;\mathrm{Prob}(X))$  the parametrized families of probability measures
	$(\bn_z)_{z\in\Omega}$ acting on $X$ which are weakly measurable with respect to $z\in\Omega.$ When $\lambda$ is the
	Lebesgue measure, we use the notation $\mathcal{P}(\lambda;X)=\mathcal{P}(\Omega;X).$
	
	The following theorem as it appears in \cite{ab97,dm87} uses the theory of generalized Young measures to describe weak 
	limits of the form
	\begin{align*}
	\lim_{n\to\infty}\int_{\Omega}\phi(x)g(u_n(x))\:dx,
	\end{align*}
	for $\phi\in C^0(\Omega),$ any bounded sequence $u_n$ in $L^1,$ and test functions $g$ such that
	\begin{align*}
	g(z)=\bar{g}(z)(1+|z|), \qquad \bar{g}\in C^b(\mathbb{R}^d).
	\end{align*}
	\begin{align}
	\label{res.fun.property}
	\begin{split}
	&\text{For $g\in\mathcal{F}_1,$ the $L^1$-recession function}\;\;
	g^{\infty}(z)=\lim_{s\to\infty}\frac{g(sz)}{s}, \;\; z\in S^{d-1},\;\; \text{coincides}\\
	&\text{with $h^{\infty}(z),$ where $h\in\mathcal{F}_0$ and $g(z)=h(z)(1+|z|).$}
	\end{split}
	\end{align}
	%%%
	\begin{theorem}\label{DP-M&A-B}
		Let $\{u_n\}$ be bounded in $L^1(\Omega;\mathbb{R}^d).$ There exists a subsequence $\{u_{n_{k}}\}$, a nonnegative Radon measure $\boldsymbol{\mu}\in\mathcal{M}^+(\Omega)$ and parametrized families of probability measures 
		\begin{align*}
		\bn\in\mathcal{P}(\Omega;\mathbb{R}^d), \qquad \bn^{\infty}\in\mathcal{P}(\lambda;S^{d-1})
		\end{align*}
		such that
		\begin{align*}
		g(u_{n_{k}}) \rightharpoonup \left\la\bn,g\right\ra+\left\la\bn^{\infty},g^{\infty}\right\ra\boldsymbol{\mu}
		\quad \text{weak-$\ast$ in}\:\:\mathcal{M}^+(\Omega),
		\end{align*}
		for any $g\in\mathcal{F}_1.$
	\end{theorem}
	Given that the only available bound for the approximate sequence $U^{\varepsilon}$ is of the form
	\begin{align*}
	\int_{\mathbb{T}^d} f(\fe,\ve,\te)\:dx<C,
	\end{align*}
	we want to represent the weak limits wk-$\ast$ $\ds\lim_{\varepsilon\to 0}f(\fe,\ve,\te),$ for a continuous test
	function $f$ satisfying the growth condition
	\begin{align*}
	|f(\fe,\ve,\te)|\leq C (1+|F|^p+|v|^2+\theta^{\ell}).
	\end{align*}
	In order to apply Theorem \ref{DP-M&A-B}, we perform the change of variables 
	\begin{align*}
	(A,b,c)=(|F|^{p-1}F,|v|v,\theta^{\ell})
	\end{align*}
	and define
	\begin{align*}
	f(F,v,\theta):=g(|F|^{p-1}F,|v|v,\theta^{\ell}),
	\end{align*}
	imposing that the function $g$ grows like
	\begin{align*}
	|g(A,b,c)|\leq C (1+|A|+|b|+|c|).
	\end{align*}
	Then Theorem~\ref{DP-M&A-B} applies to represent the wk-$\ast$ limits of $g$:
	\begin{align*}
	g^{\infty}(A,b,c)=\lim_{s\to\infty}\frac{g(sA,sb,sc)}{1+s(A,b,c)},
	\end{align*}
	for all $(A,b,c)\in S^{d^2+d}\cap\{c>0\}.$ Consequently, there exist a nonnegative Borel measure
	$M\in\mathcal{M}^+(\bar{Q}_T)$ and probability measures $N\in\mathcal{P}(\bar{Q}_T;\mathbb{R}^{d^2+d+1}),$ $N^{\infty}\in\mathcal{P}(\bar{Q}_T;S^{d^2+d})$ which up to subsequence
	\begin{align*}
	g(A_n,b_n,c_n)
	\rightharpoonup \left\la N,g(\lambda_{A},\lambda_{b},\lambda_{c})\right\ra 
	+\left\la N^{\infty},g^{\infty}(\lambda_{A},\lambda_{b},\lambda_{c})\right\ra M.
	\end{align*}
	Then, property (\ref{res.fun.property}) implies that
	\begin{align*}
	f(F_n,v_n,\theta_n)
	\rightharpoonup \left\la\bn,f(\lambda_{F},\lambda_{v},\lambda_{\theta})\right\ra 
	+\left\la\bn^{\infty},f^{\infty}(\lambda_{F},\lambda_{v},\lambda_{\theta})\right\ra M
	\end{align*}
	%	%%%
	%	{\color{red}
	%		\begin{align*}
	%		f(\Phi(F_n),v_n,\theta_n)
	%		\rightharpoonup \left\la\bn,f(\lambda_{F},\lambda_{v},\lambda_{\theta})\right\ra 
	%		+\left\la\bn^{\infty},f^{\infty}(\lambda_{F},\lambda_{v},\lambda_{\theta})\right\ra M
	%		\end{align*}
	%	}
	where 
	\begin{align*}
	\left\la\bn,f(\lambda_{F},\lambda_{v},\lambda_{\theta})\right\ra =
	\left\la N,g(|\lambda_{F}|^{p-1}\lambda_{F},|\lambda_{v}|\lambda_{v},\lambda_{\theta}^{\ell})\right\ra 
	\end{align*}
	and
	\begin{align*}
	\left\la\bn^{\infty},f(\lambda_{F},\lambda_{v},\lambda_{\theta})\right\ra =
	\left\la N,g^{\infty}(|\lambda_{F}|^{p-1}\lambda_{F},|\lambda_{v}|\lambda_{v},\lambda_{\theta}^{\ell})\right\ra .
	\end{align*}
	Therefore, given the bound (\ref{energy_unif_bound}) and assuming that the recession function
	\begin{align*}
	\left(e(F,\theta)+\frac{1}{2}|v|^2\right)^{\infty}=
	\lim_{s\to\infty}\frac{e\left(s^{1/p}F,s^{1/\ell}\theta\right)+\ds\frac{s}{2}|v|^2}
	{1+s(|F|^p,|v|^2,\theta^{\ell})},
	\end{align*}
	exists and is continuous for all $(|F|^{p-1}F,|v|v,\theta^{\ell})\in
	S^{d^2+d}\cap\{c>0\},$ we have that (along a subsequence)
	\begin{small}
		\begin{align*}
		\text{wk-$\ast$-}\lim_{\varepsilon\to 0}\left(\hat{e}(\Phi(\fe),\te)+\frac{1}{2}|\ve|^2\right)
		&=\left\la\bn,e(\lf,\lt)+\frac{1}{2}|\lv|^2\right\ra\\
		&\quad+\left\la\bn^{\infty},\left(e(\lf,\lt)+\frac{1}{2}|\lv|^2\right)^{\infty}\right\ra M\;,
		\end{align*}
	\end{small}
	recalling~(\ref{eta&e}$)_2$. Then (\ref{gr.con.1}) implies that $\left(e(\lf,\lt)+\frac{1}{2}|\lv|^2\right)^{\infty}>0,$
	therefore %for
	%\begin{equation}
	%	\label{def.gamma.F}
	%	\bg:=\left\la\bn^{\infty},\left(\frac{1}{2}|\lv|^2+e(\lf,\lt)\right)^{\infty}\right\ra M\in \mathcal{M}^+(\bar{Q}_T).
	%\end{equation}
	\begin{small}
		\begin{align}
		\begin{split}
		\label{def.gamma}
		\bg:=\left\la\bn^{\infty},\left(\frac{1}{2}|\lv|^2+e(\lf,\lt)\right)^{\infty}\right\ra M
		=\left\la\bn^{\infty},\left(\frac{1}{2}|\lv|^2+\hat{e}(\Phi(\lf),\lt)\right)^{\infty}\right\ra M
		\in\mathcal{M}^+(\bar{Q}_T).
		\end{split}
		\end{align}
	\end{small}
	%we have that $\bg\in \mathcal{M}^+(\bar{Q}_T).$
	
	\section{The averaged relative entropy inequality} \label{sec4}
	The augmented system~\eqref{sysaug} belongs to a general class of hyperbolic systems of the form 
	\begin{equation*}
	\partial_t A(U)+\partial_{\alpha}f_{\alpha}(U)=0
	\end{equation*}
	where $U=U(x,t)\in\mathbb{R}^n,$ is the unknown with $x\in\mathbb{R}^d,$ $t\in\mathbb{R}^+$
	and $A,f_{\alpha}:\mathbb{R}^n\to\mathbb{R}^n$ are given smooth functions of $U.$ 
	It is symmetrizable in the sense of Friedrichs and Lax~\cite{MR0285799}, under appropriate hypotheses: The map 
	$A(U)$ is globally invertible and there exists an entropy-entropy flux pair $(H,q)$, \textit{i.e.} there exists a 
	smooth multiplier $G(U):\mathbb{R}^{n}\to\mathbb{R}^{n}$ such that
	\begin{align*}
	\nabla H&=G\cdot\nabla A\\
	\nabla q_\alpha&=G\cdot\nabla f_\alpha,\quad \alpha=1,\dots,d.
	\end{align*}
	In our case
	\begin{small}
		$$U=\left( \begin{array}{c} \Phi(F)\\ v\\ \theta  \end{array} \right ),\;\;
		A(U)=\left( \begin{array}{c} \Phi(F)\\ v\\ \frac{1}{2} |v|^2+\hat{e}(\Phi(F), \theta) \end{array} \right ),\;\;
		f_\alpha (U)= \left( \begin{array}{c} \phv \\   \frac{\partial \hat{\psi}}{\partial\xi^B}(\Phi(F),\theta)\ph  \\ 
		\frac{\partial \hat{\psi}}{\partial\xi^B}(\Phi(F),\theta)\ph v_i  \end{array} \right )\; ,$$ 	
	\end{small}
	while the (mathematical) entropy is given by $H(U)=-\hat\eta(\Phi(F),\theta),$ the entropy flux $q_{\alpha} = 0$
	and the associated multiplier is
	$$G(U)=\frac{1}{\theta}\left(\frac{\partial\hat{\psi}}{\partial\xi^B}(\Phi(F),\theta),v,-1\right)^T,\quad 
	B=1,\dots,19$$ see \cite{christoforou2016relative,CGT2017}. 
	
	Consider a strong solution $(\Phi(\bar{F}),\bar{v},\bar{\theta})^T \in W^{1,\infty}(Q_T)$ to (\ref{sys2})
	that satisfies the entropy identity (\ref{entropy.prod}) and a dissipative measure valued solution to 
	(\ref{sys2}),~(\ref{entropy.prod}) according to Definition \ref{def.mv.F}. We write the difference of the weak form of 
	equations (\ref{sys2}),~(\ref{entropy.prod}) and (\ref{mv.sol.Fvt}),~(\ref{mv.sol.Fvt.energy}) to obtain the following 
	three integral identities
	\begin{align}
	\begin{split}
	\int(\Phi^B(F)-\Phi^B(\bar{F}))(x,0)&\phi_1(x,0)\:dx +\int_0^T \int(\Phi^B(F)-\Phi^B(\bar{F}))\partial_t\phi_1(x,t)\:dx\:dt\\
	&=\int_0^T \int \left(\phv-\phbv\right)\partial_{\alpha}\phi_1(x,t)\:dx\:dt,
	\end{split}
	\label{eq1.wk.mv}
	\end{align}
	%\vfil\eject
	\begin{align}
	\begin{split}
	\int&(\left\la\bn,\lvi\right\ra-\bar{v}_i)(x,0)\phi_2(x,0)\:dx
	+\int_0^T\int (\left\la\bn,\lvi\right\ra-\bar{v}_i)\partial_t\phi_2(x,t)\:dx\:dt\\
	&=\int_0^T\int \left(\left\la\bn,\frac{\partial\hat{\psi}}{\partial\xi^B}(\Phi(\lambda_{F}),\lambda_{\theta})\phl\right\ra
	-\frac{\partial\hat{\psi}}{\partial\xi^B}(\Phi(\bar{F}),\bar{\theta})\phb\right)\partial_{\alpha}\phi_2(x,t)\:dx\:dt\;,
	\end{split}
	\label{eq2.wk.mv}
	\end{align}
	and
	\begin{align}
	\label{eq3.wk.mv}
	\begin{split}
	\int&\left(\left\la\bn,\frac{1}{2}|\lv|^2+\hat{e}(\Phi(\lf),\lt)\right\ra
	-\frac{1}{2}|\bar{v}|^2-\hat{e}(\Phi(\bar{F}),\bar{\theta})\right)\!(x,0)\:\phi_3(x,0)\:dx \\
	&\;+\int_0^T \int \left\{\left(\left\la\bn,\frac{1}{2}|\lv|^2+\hat{e}(\Phi(\lf),\lt)\right\ra
	-\frac{1}{2}|\bar{v}|^2-\hat{e}(\Phi(\bar{F}),\bar{\theta})\right)+\bg\right\}\partial_t\phi_3(x,t)\:dx\:dt\\
	&=-\int_0^T \int (\left\la\bn,r\right\ra-\bar{r})\phi_3(x,t)\:dx\:dt,
	\end{split}
	\end{align}
	for any $\phi_i\in C^1_c(Q_T)$, $i=1, 2$ and $\phi_3\in C^1_c[0,T).$ Similarly, testing the difference of 
	(\ref{entropy.prod}) and (\ref{mv.sol.Fvt})$_3$ against $\phi_4\in C^1_c(Q_T),$  with $\phi_4\ge 0$, we have
	%{\color{red}
	\begin{align}
	\label{entr1.wka.mv}
	\begin{split}
	-\int&(\left\la\bn,\hat{\eta}(\Phi(\lf),\lt)\right\ra-\hat{\eta}(\Phi(\bar{F}),\bar{\theta}))(x,0)\phi_4(x,0)\:dx\\
	&\quad-\int_0^T\int(\left\la\bn,\hat{\eta}(\Phi(\lf),\lt)\right\ra-\hat{\eta}(\Phi(\bar{F}),\bar{\theta}))\partial_t\phi_4(x,t)\:dx\:dt\\ &\geq\int_0^T\int\left(\left\la\bn,\frac{r}{\lt}\right\ra-\frac{\bar{r}}{\bar{\theta}}\right)\phi_4(x,t)\:dx\:dt.
	\end{split}
	\end{align}
	We then choose $(\phi_1,\phi_2,\phi_3)=-\bar\theta\,G(\bar{U})\varphi(t)=( -\frac{\partial\hat{\psi}}{\partial\xi^B}(\Phi(\bar{F}),\bar{\theta}) ,-\bar{v},1)^T\varphi(t)$, 
	for some $\varphi\in C_c^1[0,T]$, thus (\ref{eq1.wk.mv}), (\ref{eq2.wk.mv}) and (\ref{eq3.wk.mv}) become
	\begin{small}
		\begin{align}
		\begin{split}
		\label{eq1phi.wk.mv}
		\int &\left(-\frac{\partial\hat{\psi}}{\partial\xi^B}(\Phi(\bar{F}),\bar{\theta})(\Phi^B(F)-\Phi^B(\bar{F}))\right)(x,0)\varphi(0)\:dx \\
		&\qquad+\int_0^T\int\left(-\frac{\partial\hat{\psi}}{\partial\xi^B}(\Phi(\bar{F}),\bar{\theta})(\Phi^B(F)-\Phi^B(\bar{F}))\right)\varphi^{\prime}(t)\:dx\:dt\\
		&=\int_0^T\int \left[\partial_t\Big(\frac{\partial\hat{\psi}}{\partial\xi^B}(\Phi(\bar{F}),\bar{\theta})\Big)(\Phi^B(F)-\Phi^B(\bar{F}))\right.\\
		&\qquad\quad\left.-\partial_{\alpha}\Big(\frac{\partial\hat{\psi}}{\partial\xi^B}(\Phi(\bar{F}),\bar{\theta})\Big)\Big(\phv-\phbv\Big)\right]\varphi(t) dxdt \, ,
		\end{split}
		\end{align}
	\end{small}
	\begin{align}
	\label{eq2phi.wk.mv}
	\begin{split}
	\int&(-\bar{v_i}(\left\la\bn,\lvi\right\ra-\bar{v_i}))(x,0)\varphi(0)\:dx 
	+\int_0^T\int-\bar{v_i}(\left\la\bn,\lvi\right\ra-\bar{v_i})\varphi^{\prime}(t)\:dx\:dt\\
	&=-\int_0^T\int\left[-\partial_{\alpha}\left(\frac{\partial\hat{\psi}}{\partial\xi^B}(\Phi(\bar{F}),\bar{\theta})\phb\right)(\left\la\bn,\lvi\right\ra-\bar{v}_i)\right.\\
	&\:\:\left.+\partial_{\alpha}\bar{v_i}\left(\left\la\bn,\frac{\partial\hat{\psi}}{\partial\xi^B}(\Phi(\lambda_{F}),\lambda_{\theta})\phl\right\ra
	-\frac{\partial\hat{\psi}}{\partial\xi^B}(\Phi(\bar{F}),\bar{\theta})\phb\right)\right]\varphi(t)\:dx\:dt\;,
	\end{split}
	\end{align} 
	and
	\begin{align}
	\begin{split}
	\label{eq3phi.wk.mv}
	\int&\left(\left\la\bn,\frac{1}{2}|\lv|^2+\hat{e}(\Phi(\lf),\lt)\right\ra
	-\frac{1}{2}|\bar{v}|^2-\hat{e}(\Phi(\bar{F}),\bar{\theta})\right)\!(x,0)\;\varphi(0)\:dx \\
	&\quad+\int_0^T \int \left\{\left(\left\la\bn,\frac{1}{2}|\lv|^2+\hat{e}(\Phi(\lf),\lt)\right\ra
	-\frac{1}{2}|\bar{v}|^2-\hat{e}(\Phi(\bar{F}),\bar{\theta})\right)+\bg\right\}\varphi^{\prime}(t)\:dx\:dt\\
	&=-\int_0^T \int (\left\la\bn,r\right\ra-\bar{r})\varphi(t)\:dx\:dt.
	\end{split}
	\end{align}
	For inequality (\ref{entr1.wka.mv}), we choose accordingly $\phi_4:=\bar{\theta}\varphi(t)\geq 0$, $\varphi \geq 0$ so that
	\begin{align}
	\begin{split}
	\label{entr2.wk.mv}
	-\int\bar{\theta}&(\left\la\bn,\hat{\eta}(\Phi(\lf),\lt)\right\ra-\hat{\eta}(\Phi(\bar{F}),\bar{\theta}))(x,0)\varphi(0)\:dx\\
	&\quad-\int_0^T\int\bar{\theta}(\left\la\bn,\hat{\eta}(\Phi(\lf),\lt)\right\ra-\hat{\eta}(\Phi(\bar{F}),\bar{\theta}))\varphi^{\prime}(t)\:dx\:dt\\
	&\geq\int_0^T\int \left[\partial_t\bar{\theta}(\left\la\bn,\hat{\eta}(\Phi(\lf),\lt)\right\ra-\hat{\eta}(\Phi(\bar{F}),\bar{\theta}))
	+\bar{\theta}\left(\left\la\bn,\frac{r}{\lt}\right\ra-\frac{\bar{r}}{\bar{\theta}}\right)\right]\varphi(t)\:dx\:dt.
	\end{split}
	\end{align}
	
	Adding together (\ref{eq1phi.wk.mv}), (\ref{eq2phi.wk.mv}), (\ref{eq3phi.wk.mv}) and (\ref{entr2.wk.mv}), we obtain the integral inequality
	\begin{small}
		\begin{align}
		\label{entr.eq.1.mv}
		%\begin{split}
		\int&\varphi(0)\bigg[-\frac{\partial\hat{\psi}}{\partial\xi^B}(\Phi(\bar{F}),\bar{\theta})(\Phi^B(F)-\Phi^B(\bar{F}))(x,0)-\la\bn,\bar{v}_i(\lvi-\bar{v}_i)\ra(x,0) \nonumber \\
		&\qquad
		+\left\la\bn,\frac{1}{2}|\lv|^2+\hat{e}(\Phi(\lf),\lt)-\frac{1}{2}|\bar{v}|^2-\hat{e}(\Phi(\bar{F}),\bar{\theta})\right\ra(x,0) \nonumber \\
		&\qquad\qquad\qquad\qquad\qquad\qquad
		-\bar{\theta}\la\bn,\hat{\eta}(\Phi(\lf),\lt)-\hat{\eta}(\Phi(\bar{F}),\bar{\theta})\ra(x,0)\bigg]\:dx \nonumber \\
		&+\int_{0}^{T}\!\!\!\!\int\varphi'(t)\bigg[-\frac{\partial\hat{\psi}}{\partial\xi^B}(\Phi(\bar{F}),\bar{\theta})(\Phi^B(F)-\Phi^B(\bar{F}))-\la\bn,\bar{v}_i(\lvi-\bar{v}_i)\ra \nonumber \\
		&\qquad\qquad
		+\left\la\bn,\frac{1}{2}|\lv|^2+\hat{e}(\Phi(\lf),\lt)-\frac{1}{2}|\bar{v}|^2-\hat{e}(\Phi(\bar{F}),\bar{\theta})\right\ra \nonumber \\
		&\qquad\qquad\qquad\qquad\qquad\qquad
		-\bar{\theta}\la\bn,\hat{\eta}(\Phi(\lf),\lt)-\hat{\eta}(\Phi(\bar{F}),\bar{\theta})\ra+\bg\bigg]\:dx\:dt \nonumber \\
		&\!\!\!\!\!\!
		\geq-\int_{0}^{T}\!\!\!\!\int\varphi(t)\bigg[-\partial_t\left(\frac{\partial\hat{\psi}}{\partial\xi^B}(\Phi(\bar{F}),\bar{\theta})\right)\!\!(\Phi^B(F)-\Phi^B(\bar{F})) \nonumber \\
		&\qquad\qquad\qquad
		+\partial_{\alpha}\left(\frac{\partial\hat{\psi}}{\partial\xi^B}(\Phi(\bar{F}),\bar{\theta})\right)\left(\phv-\phbv\right) \nonumber \\
		&\qquad\qquad\qquad
		+\partial_{\alpha}\bar{v}_i\left\la\bn,\frac{\partial\hat{\psi}}{\partial\xi^B}(\Phi(\lambda_{F}),\lambda_{\theta})\phl
		-\frac{\partial\hat{\psi}}{\partial\xi^B}(\Phi(\bar{F}),\bar{\theta})\phb\right\ra \nonumber \\
		&\qquad\qquad\qquad
		-\partial_t\bar{\theta}\la\bn,\hat{\eta}(\Phi(\lf),\lt)-\hat{\eta}(\Phi(\bar{F}),\bar{\theta})\ra \nonumber \\
		&\qquad\qquad\qquad
		-\partial_{\alpha}\left(\frac{\partial\hat{\psi}}{\partial\xi^B}(\Phi(\bar{F}),\bar{\theta})\phb\right)\la\bn,(\lvi-\bar{v}_i)\ra
		-\bar{\theta}\left\la\bn,\frac{r}{\lt}-\frac{\bar{r}}{\bar{\theta}}\right\ra+\la\bn,r-\bar{r}\ra\bigg]\:dx\:dt\nonumber \\
		&\!\!\!\!\!\!
		=:-\int_{0}^{T}\int\varphi(t)K(x,t)\:dx\:dt\;.
		%\end{split}
		\end{align}
	\end{small}
	%\left\la\bn,\right\ra|_{t=0}
	Using the entropy identity (\ref{entropy.prod}) and the null-Lagrangian property (\ref{null-Lan}), the quantity 
	$K(x,t)$ in the integrand on the right hand-side of (\ref{entr.eq.1.mv}) becomes
	\begin{small}
		\begin{align}
		\label{entr.eq.2.mv}
		K&=-\partial_t\bar{\theta}\left\la\bn,\hat{\eta}(\Phi(\lf),\lt|\Phi(\bar{F}),\bar{\theta})\right\ra
		+\partial_t\Phi^B(\bar{F})\left\la\bn,\frac{\partial\hat{\psi}}{\partial \xi^B}(\Phi(\lf),\lt|\Phi(\bar{F}),\bar{\theta})\right\ra \nonumber\\
		%-\partial_t\bar{\theta}\left(\la\bn,\hat{\eta}(\Phi(\lf),\lt)-\hat{\eta}(\Phi(\bar{F}),\bar{\theta})\ra
		%+\frac{\partial^2 \hat{\psi}}{\partial\xi^B\partial\theta}(\Phi(\bar{F}),\bar{\theta})(\Phi^B(F)-\Phi^B(\bar{F}))
		%-\frac{\partial\hat{\eta}}{\partial\theta}(\Phi(\bar{F}),\bar{\theta})\la\bn,\lt-\bar{\theta}\ra\right)\\
		&-\partial_t\bar{\theta}\:\frac{\partial\hat{\eta}}{\partial\theta}(\Phi(\bar{F}),\bar{\theta})\la\bn,\lt-\bar{\theta}\ra \nonumber\\
		%&+\partial_t\Phi^B(\bar{F})\left(\left\la\bn,\frac{\partial\hat{\psi}}{\partial\xi^B}(\Phi(\lambda_{F}),\lambda_{\theta})
		%-\frac{\partial\hat{\psi}}{\partial\xi^B}(\Phi(\bar{F}),\bar{\theta})\right\ra
		%-\frac{\partial^2 \hat{\psi}}{\partial \xi^B\xi^A}(\Phi(\bar{F}),\bar{\theta})(\Phi^B(F)-\Phi^B(\bar{F}))\right.\nonumber\\
		%&\left.-\frac{\partial^2 \hat{\psi}}{\partial\xi^B\partial\theta}(\Phi(\bar{F}),\bar{\theta})\la\bn,\lt-\bar{\theta}\ra\!\right)\\
		&-\partial_t\Phi^B(\bar{F})\left(\left\la\bn,\frac{\partial\hat{\psi}}{\partial\xi^B}(\Phi(\lambda_{F}),\lambda_{\theta})
		-\frac{\partial\hat{\psi}}{\partial\xi^B}(\Phi(\bar{F}),\bar{\theta})\right\ra
		-\frac{\partial^2 \hat{\psi}}{\partial\xi^B\partial\theta}(\Phi(\bar{F}),\bar{\theta})\la\bn,\lt-\bar{\theta}\ra\!\right)\nonumber\\
		&+\partial_{\alpha}\left(\frac{\partial\hat{\psi}}{\partial\xi^B}(\Phi(\bar{F}),\bar{\theta})\right)\!\!\left(\phv-\phbv\right)
		-\partial_{\alpha}\left(\frac{\partial\hat{\psi}}{\partial\xi^B}(\Phi(\bar{F}),\bar{\theta})\phb\right)\la\bn,(\lvi-\bar{v}_i)\ra\nonumber\\
		&+\partial_{\alpha}\bar{v}_i\left\la\bn,\frac{\partial\hat{\psi}}{\partial\xi^B}(\Phi(\lambda_{F}),\lambda_{\theta})\phl
		-\frac{\partial\hat{\psi}}{\partial\xi^B}(\Phi(\bar{F}),\bar{\theta})\phb\right\ra
		-\bar{\theta}\left\la\bn,\frac{r}{\lt}-\frac{\bar{r}}{\bar{\theta}}\right\ra+\la\bn,r-\bar{r}\ra\nonumber\\
		&\!\!\!\!\!\!=-\partial_t\bar{\theta}\left\la\bn,\hat{\eta}(\Phi(\lf),\lt|\Phi(\bar{F}),\bar{\theta})\right\ra
		+\partial_t\Phi^B(\bar{F})\left\la\bn,\frac{\partial\hat{\psi}}{\partial \xi^B}(\Phi(\lf),\lt|\Phi(\bar{F}),\bar{\theta})\right\ra
		-\frac{\bar{r}}{\bar{\theta}}\left\la\bn,\lt-\bar{\theta}\right\ra\nonumber\\
		&-\phb\partial_{\alpha}\bar{v}_i\left\la\bn,\frac{\partial\hat{\psi}}{\partial\xi^B}(\Phi(\lambda_{F}),\lambda_{\theta})
		-\frac{\partial\hat{\psi}}{\partial\xi^B}(\Phi(\bar{F}),\bar{\theta})\right\ra\nonumber\\
		&+\partial_{\alpha}\left(\frac{\partial\hat{\psi}}{\partial\xi^B}(\Phi(\bar{F}),\bar{\theta})\right)\left(\phv-\phbv\right)\nonumber\\
		&-\phb\partial_{\alpha}\!\left(\frac{\partial\hat{\psi}}{\partial\xi^B}(\Phi(\bar{F}),\bar{\theta})\right)\la\bn,(\lvi-\bar{v}_i)\ra\nonumber\\
		&+\partial_{\alpha}\bar{v}_i\left\la\bn,\frac{\partial\hat{\psi}}{\partial\xi^B}(\Phi(\lambda_{F}),\lambda_{\theta})\phl
		-\frac{\partial\hat{\psi}}{\partial\xi^B}(\Phi(\bar{F}),\bar{\theta})\phb\right\ra\nonumber\\
		&-\bar{\theta}\left\la\bn,\frac{r}{\lt}-\frac{\bar{r}}{\bar{\theta}}\right\ra+\la\bn,r-\bar{r}\ra
		\end{align}
	\end{small}
	employing (\ref{sys2})$_1$ and (\ref{null-Lan}). Here, we use the quantites
	\begin{small}
		\begin{align}
		\begin{split}
		\label{partial.eta.rel.mv}
		\left\la\bn,\hat{\eta}(\Phi(\lf),\lt|\Phi(\bar{F}),\bar{\theta})\right\ra
		&:=\Bigg\langle\bn,\hat{\eta}(\Phi(\lf),\lt)-\hat{\eta}(\Phi(\bar{F}),\bar{\theta})\\
		&\!\!\!\!\!\!\!\!\!\!\!\!\!\!\!\!\!\!\!\!
		-\frac{\partial \hat{\eta}}{\partial\xi^B}(\Phi(\bar{F}),\bar{\theta})(\Phi^B(\lf)-\Phi^B(\bar{F}))
		-\frac{\partial\hat{\eta}}{\partial\theta}(\Phi(\bar{F}),\bar{\theta})(\lt-\bar{\theta})\Bigg\rangle\;,
		\end{split}
		\end{align}
		and
		\begin{align}
		\begin{split}
		\label{partial.g.rel.mv}
		\left\la\bn,\frac{\partial\hat{\psi}}{\partial \xi^B}(\Phi(\lf),\lt|\Phi(\bar{F}),\bar{\theta})\right\ra
		&:=\Bigg\langle\bn,\frac{\partial\hat{\psi}}{\partial\xi^B}(\Phi(\lambda_{F}),\lambda_{\theta})
		-\frac{\partial\hat{\psi}}{\partial\xi^B}(\Phi(\bar{F}),\bar{\theta})\\
		&\!\!\!\!\!\!\!\!\!\!\!\!\!\!\!\!\!\!\!\!!\!\!\!\!\!\!\!
		-\frac{\partial^2 \hat{\psi}}{\partial \xi^B\xi^A}(\Phi(\bar{F}),\bar{\theta})(\Phi^B(\lf)-\Phi^B(\bar{F}))
		-\frac{\partial^2 \hat{\psi}}{\partial\xi^B\partial\theta}(\Phi(\bar{F}),\bar{\theta})(\lt-\bar{\theta})\Bigg\rangle.
		\end{split}
		\end{align}
	\end{small}
	Next, we rewrite the terms
	\begin{small}
		\begin{align}
		\label{entr.eq.3.mv}
		&-\phb\partial_{\alpha}\bar{v}_i\left\la\bn,\frac{\partial\hat{\psi}}{\partial\xi^B}(\Phi(\lambda_{F}),\lambda_{\theta})
		-\frac{\partial\hat{\psi}}{\partial\xi^B}(\Phi(\bar{F}),\bar{\theta})\right\ra\nonumber\\
		&+\partial_{\alpha}\left(\frac{\partial\hat{\psi}}{\partial\xi^B}(\Phi(\bar{F}),\bar{\theta})\right)\left(\phv-\phbv\right)\nonumber\\
		&-\phb\partial_{\alpha}\!\left(\frac{\partial\hat{\psi}}{\partial\xi^B}(\Phi(\bar{F}),\bar{\theta})\right)\la\bn,(\lvi-\bar{v}_i)\ra\nonumber\\
		&+\partial_{\alpha}\bar{v}_i\left\la\bn,\frac{\partial\hat{\psi}}{\partial\xi^B}(\Phi(\lambda_{F}),\lambda_{\theta})\phl
		-\frac{\partial\hat{\psi}}{\partial\xi^B}(\Phi(\bar{F}),\bar{\theta})\phb\right\ra\nonumber\\
		&\;=\partial_{\alpha}\bar{v}_i\left\la\bn,\left(\frac{\partial\hat{\psi}}{\partial\xi^B}(\Phi(\lambda_{F}),\lambda_{\theta})
		-\frac{\partial\hat{\psi}}{\partial\xi^B}(\Phi(\bar{F}),\bar{\theta})\right)\left(\phl-\phb\right)\right\ra\nonumber\\
		&\quad\quad\;\;\;+\partial_{\alpha}\left(\frac{\partial\hat{\psi}}{\partial\xi^B}(\Phi(\bar{F}),\bar{\theta})\right)\!\left(\ph-\phb\right)\!\la\bn,(\lvi-\bar{v}_i)\ra\nonumber\\
		&\quad\quad\;\;\;+\partial_{\alpha}\left(\frac{\partial\hat{\psi}}{\partial\xi^B}(\Phi(\bar{F}),\bar{\theta})\left(\ph-\phb\right)\bar{v}_i\right)\;,
		\end{align}
	\end{small}
	since there holds $$\ph=\left\la\bn,\phl\right\ra$$ and because of the null-Lagrangian property (\ref{null-Lan}).
	Also, we observe that
	\begin{align}
	\label{entr.eq.4.mv}
	-\frac{\bar{r}}{\bar{\theta}}\left\la\bn,\lt-\bar{\theta}\right\ra
	-\bar{\theta}\left\la\bn,\frac{r}{\lt}-\frac{\bar{r}}{\bar{\theta}}\right\ra+\la\bn,r-\bar{r}\ra
	=\left\la\bn,\left(\frac{r}{\lt}-\frac{\bar{r}}{\bar{\theta}}\right)(\lt-\bar{\theta})\right\ra\;.
	\end{align}
	Finally, if we define the averaged quantity
	\begin{align}
	\begin{split}
	\label{def.I}
	I(\lambda_{U}&|\bar{U})=I(\Phi(\lambda_{F}),\lv,\lt|\Phi(\bar{F}),\bar{v},\bar{\theta})\\
	&:=\hat{\psi}(\Phi(\lf),\lt|\Phi(\bar{F}),\bar{\theta})
	+(\hat{\eta}(\Phi(\lf),\lt)-\hat{\eta}(\Phi(\bar{F}),\bar{\theta})(\lt-\bar{\theta})+\frac{1}{2}|\lv-\bar{v}|^2,
	\end{split}
	\end{align}
	for
	%\begin{small}
		\begin{align*}
		\hat{\psi}(\Phi(\lf),\lt|\Phi(\bar{F}),\bar{\theta})
		:=&\hat{\psi}(\Phi(\lf),\lt)-\hat{\psi}(\Phi(\bar{F}),\bar{\theta})
		-\frac{\partial\hat{\psi}}{\partial\xi^B}(\Phi(\bar{F}),\bar{\theta})(\Phi^B(\lf)-\Phi^B(\bar{F}))\\
		&-\frac{\partial\hat{\psi}}{\partial\theta}(\Phi(\bar{F}),\bar{\theta})(\lt-\bar{\theta}),
		\end{align*}
	%\end{small}
	and then combine (\ref{entr.eq.1.mv}),(\ref{entr.eq.2.mv}),(\ref{entr.eq.3.mv}) and (\ref{entr.eq.4.mv}), we arrive at the 
	relative entropy inequality
	\begin{small}
		\begin{align}
		\label{rel.en.id.mv}
		\begin{split}
		&\int\varphi(0)[\left\la\bn,I(\lambda_{U_0}|\bar{U_0})\right\ra\:dx]
		+\int_{0}^{T}\int\varphi'(t)\left[\left\la\bn,I(\lambda_{U}|\bar{U})\right\ra\:dx\:dt+\bg(dx\,dt)\right]\\
		&\!\!\!\!\!\!\!\!\!\!\!\!\geq-\int_{0}^{T}\int\varphi(t)\left[-\partial_t\bar{\theta}\left\la\bn,\hat{\eta}(\Phi(\lf),\lt|\Phi(\bar{F}),\bar{\theta})\right\ra
		+\partial_t\Phi^B(\bar{F})\left\la\bn,\frac{\partial\hat{\psi}}{\partial \xi^B}(\Phi(\lf),\lt|\Phi(\bar{F}),\bar{\theta})\right\ra\right.\\
		&\!\!\!\!\!\!\!\!\!\!\left.+\partial_{\alpha}\bar{v}_i\left\la\bn,\left(\frac{\partial\hat{\psi}}{\partial\xi^B}(\Phi(\lambda_{F}),\lambda_{\theta})
		-\frac{\partial\hat{\psi}}{\partial\xi^B}(\Phi(\bar{F}),\bar{\theta})\right)\left(\phl-\phb\right)\right\ra\right.\\
		&\!\!\!\!\!\!\!\!\!\!\left.+\partial_{\alpha}\left(\frac{\partial\hat{\psi}}{\partial\xi^B}(\Phi(\bar{F}),\bar{\theta})\right)\!\left(\ph-\phb\right)\!\la\bn,(\lvi-\bar{v}_i)\ra
		+\left\la\bn,\left(\frac{r}{\lt}-\frac{\bar{r}}{\bar{\theta}}\right)(\lt-\bar{\theta})\right\ra\right]\:dx\,dt.
		\end{split}
		\end{align}
	\end{small}
	We note that the last term in (\ref{entr.eq.3.mv}) vanishes when we substitute into the integral relation
	(\ref{entr.eq.1.mv}).
	
	\section{Uniqueness of smooth solutions in the class of dissipative measure-valued solutions} \label{sec5}
	
	In this section, we state and prove the main theorem on dissipative measure-valued versus strong uniqueness. Before we 
	proceed with the proof, we show some useful estimates on the terms appearing in (\ref{rel.en.id.mv}), which yield the 
	relative entropy as a “metric” measuring the distance between the two solutions. These bounds are obtained by using 
	the convexity of the free energy function in the compact domain and the growth conditions 
	(\ref{gr.con.1})-(\ref{gr.con.4}) placed on the constitutive functions in the original variables, in the unbounded 
	domain.
	\begin{lemma}
		\label{lemma1}
		Assume that $(\bar{F},\bar{v},\bar{\theta})$ are defined in the compact set
		\begin{equation}
		\Gamma_{M,\delta}:=\left\{(\bar{F},\bar{v},\bar{\theta}): |\bar{F}|\leq M, |\bar{v}|\leq M, \;\; 0<\delta\leq\bar{\theta}\leq M\right\}
		\label{defGamma}
		\end{equation}
		for some positive constants $M$ and $\delta$ 
		and let $\hat{\psi}=\hat{e}-\theta\hat{\eta}\in C^2(\mathbb{R}^{19}\times[0,\infty)).$ 
		Assuming the growth conditions (\ref{gr.con.1})-(\ref{gr.con.4}) and $p>3,\:\ell>1,$ then
		there exist $R=R(M,\delta)$ and constants $K_1=K_1(M,\delta,c)>0,\:K_2=K_2(M,\delta,c)>0$ such that
		\begin{small}
			\begin{align}
			\label{bound4}
			\begin{split}
			I(\Phi(F),v,\theta|\Phi(\bar{F}),\bar{v},\bar{\theta})\geq
			\begin{dcases}
			\frac{K_1}{2}(|F|^p+\theta^{\ell}+|v|^2), \qquad\qquad\qquad\quad\quad\quad |F|^p\!+\theta^{\ell}\!+|v|^2>R\\
			K_2(|\Phi(F)-\Phi(\bar{F})|^2\!+\!|\theta-\bar{\theta}|^2\!+\!|v-\bar{v}|^2), \quad |F|^p\!+\theta^{\ell}\!+|v|^2\leq R
			\end{dcases}
			\end{split}
			\end{align}
		\end{small}
		for all $(\bar{F},\bar{v},\bar{\theta})\in\Gamma_{M,\delta}$.
	\end{lemma}
	\begin{proof}
		Let $(\bar{F},\bar{v},\bar{\theta})\in\Gamma_{M,\delta}$ and choose $r=r(M):=M^p+M^{\ell}+M^2$
		for which $\Gamma_{M,\delta}\subset B_r=\{(F,v,\theta):\; |F|^p+\theta^{\ell}+|v|^2\leq r\}$. 
		Taking under consideration (\ref{eta&e}), we can write $I(\Phi(F),v,\theta|\Phi(\bar{F}),\bar{v},\bar{\theta})$ in the form
		\begin{align*}
		I&(\Phi(F),v,\theta|\Phi(\bar{F}),\bar{v},\bar{\theta})\\
		&=\hat{e}(\Phi(F),\theta)-\hat{\psi}(\Phi(\bar{F}),\bar{\theta})-\frac{\partial\hat{\psi}}{\partial\xi}(\Phi(\bar{F}),\bar{\theta})\cdot(\Phi(F)-\Phi(\bar{F}))
		-\bar{\theta}\hat{\eta}(\Phi(F),\theta)+\frac{1}{2}|v-\bar{v}|^2\\
		&=e(F,\theta)-\psi(\bar{F},\bar{\theta})-\frac{\partial\hat{\psi}}{\partial\xi}(\Phi(\bar{F}),\bar{\theta})\cdot(\Phi(F)-\Phi(\bar{F}))
		-\bar{\theta}\eta(F,\theta)+\frac{1}{2}|v-\bar{v}|^2\:.
		\end{align*}
		Using (\ref{gr.con.1}) and (\ref{gr.con.3}) we have
		\begin{align*}
		I(\Phi(F),v,\theta|\Phi(\bar{F}),\bar{v},\bar{\theta})
		\geq c(|F|^p+\theta^{\ell}) -c_1 -c_2|F|^3 -c_3|\eta(F,\theta)| +\frac{1}{2}|v|^2 -c_4|v|.
		\end{align*}
		Selecting now $R$ sufficiently large such that $R>r(M)+1$ and for $|F|^p+\theta^{\ell}+|v|^2>R$ we have
		\begin{align*}
		I(\Phi(F),v,\theta|\Phi(\bar{F}),\bar{v},\bar{\theta})
		&\geq \min\left\{c,\frac{1}{2}\right\} (|F|^p+\theta^{\ell}+|v|^2)-c_3|\eta(F,\theta)|-c_5\\
		&\geq \frac{1}{2}\min\left\{c,\frac{1}{2}\right\} (|F|^p+\theta^{\ell}+|v|^2)
		\end{align*}
		and (\ref{bound4}) is established within the region $|F|^p+\theta^{\ell}+|v|^2>R$.
		
		In the complementary region $|F|^p+\theta^{\ell}+|v|^2\leq R,$ observe that $(F,\mathrm{cof}F,\det F,v,\theta)$ takes values in the set
		\begin{small}
			\begin{equation*}
			%\label{defD}
			D:=\left\{(\Phi(F),v,\theta): |F|\leq R^{1/p}, |\mathrm{cof}F|\leq CR^{2/p}, |\det F|\leq CR^{3/p}, |v|\leq R^{1/2}, 0<\theta\leq R^{1/\ell}\right\},
			\end{equation*}
		\end{small}
		for some constant $C.$
		We use the convexity of the entropy $\tilde{H}(V)$ in the symmetric variables 
		$V:=A(U)=\left(\xi,v,\frac{|v|^2}{2}+\hat{e}(\xi,\theta)\right)^T:$ 
		\begin{align*}
		\frac{1}{\bar{\theta}}I(\xi,v,\theta|\bar{\xi},\bar{v},\bar{\theta})
		&=\tilde{H}(A(U)|A(\bar{U}))\\
		&=\tilde{H}(A(U))-\tilde{H}(A(\bar{U}))-\tilde{H}_V(A(\bar{U}))(A(U)-A(\bar{U}))\\
		&\geq\min_{V^*\in D^*}\{\tilde{H}_{VV} (V^*)\}  |A(U)-A(\bar{U})|^2,
		\end{align*}
		since $\tilde{H}(V)$ is convex in $V$ and $D^*$ is the compact domain determined by the map $V=A(U)$ and the
		set $D$ defined above. Moreover, using the invertibility at the map $U\mapsto A(U)$
		\begin{align*}
		|U-\bar{U}|&=\left|\int_0^1\frac{d}{d\tau}[A^{-1}(\tau A(U)+(1-\tau)A(\bar{U}))]\:d\tau\right|\\
		&\leq\left|\int_0^1\nabla_V(A^{-1})(\tau A(U)+(1-\tau)A(\bar{U}))\:d\tau\right|\:|A(U)-A(\bar{U})|\\
		&\leq C' |A(U)-A(\bar{U})|,
		\end{align*}
		where
		$$C'=\sup_{\substack{U\in B_R \\ \bar{U}\in\Gamma_{M,\delta}}}\left|\int_0^1\nabla_V(A^{-1})(\tau A(U)+(1-\tau)A(\bar{U}))\:d\tau\right|\:|A(U)-A(\bar{U})|<\infty.$$
		Therefore
		\begin{align*}
		I(\Phi(F),v,\theta|\Phi(\bar{F}),\bar{v},\bar{\theta})\geq\frac{K_2}{C'}|U-\bar{U}|^2,
		\end{align*}
		for $K_2:=\displaystyle\delta\min\{\tilde{H}_{VV}(  V^*)\}>0$ and the proof is complete.
	\end{proof}
	%%%%%%%%%%%%%%%%%%%%%%%%%%%%%%%%%%%%%%%%%%%%%%%%%%%%%%%%%%%%%%%%%%%%%%%%%%%%%%%%%%%%%%%%%%%%%%%%%%%%%%%%%%%%%%%%%%%%%%
	\begin{lemma}
		\label{lemma2}
		Under the assumptions of Lemma \ref{lemma1} and the additional growth hypothesis 
		\begin{align}
		\label{gr.con.5}
		\left|\frac{\partial\hat{\psi}}{\partial\xi}(\xi,\theta)\right|\leq c\:|\hat{\psi}(\xi,\theta)|,\qquad \forall\,\xi,\,\theta
		\end{align}
		for some positive constant $c,$ the following bounds hold true:
		\begin{enumerate}[label=(\roman*)]
			\item
			There exist constants $C_1,C_2,C_3,C_4>0$ such that
			\begin{small}
				\begin{equation}
				\label{bound6}
				\left|\left(\frac{\partial\Phi^B}{\partial F_{i\alpha}}(F)
				-\frac{\partial\Phi^B}{\partial F_{i\alpha}}(\bar{F})\right)
				\left(\frac{\partial\hat{\psi}}{\partial\xi^B}(\Phi(F),\theta)
				-\frac{\partial\hat{\psi}}{\partial\xi^B}(\Phi(\bar{F}),\bar{\theta})\right)\right|
				\leq C_1 I(\Phi(F),v,\theta|\Phi(\bar{F}),\bar{v},\bar{\theta})\;,
				\end{equation}
			\end{small}
			\begin{equation}
			\label{bound7}
			\left|\frac{\partial\hat{\psi}}{\partial\xi}(\Phi(F),\theta|\Phi(\bar{F}),\bar{\theta})\right|\leq C_2 I(\Phi(F),v,\theta|\Phi(\bar{F}),\bar{v},\bar{\theta})\;,
			\end{equation}
			\begin{equation}
			\label{bound5}
			|\hat{\eta}(\Phi(F),\theta|\Phi(\bar{F}),\bar{\theta})|\leq C_3 I(\Phi(F),v,\theta|\Phi(\bar{F}),\bar{v},\bar{\theta})\;,
			\end{equation}
			and 
			\begin{equation}
			\label{bound8}
			\left|\left(\frac{\partial\Phi^B}{\partial F_{i\alpha}}(F)
			-\frac{\partial\Phi^B}{\partial F_{i\alpha}}(\bar{F})\right)(v_i-\bar{v}_i)\right|
			\leq C_4 I(\Phi(F),v,\theta|\Phi(\bar{F}),\bar{v},\bar{\theta})
			\end{equation}
			for all $(\bar{F},\bar{v},\bar{\theta})\in\Gamma_{M,\delta}$.
			\item
			There exist constants $K_1^{\prime},\:K_2^{\prime}$ and $R>0$ sufficiently large such that
			\begin{small}
				\begin{align}
				\label{bound9}
				\begin{split}
				I(\Phi(F),v,\theta|\Phi(\bar{F}),\bar{v},\bar{\theta})\geq
				\begin{dcases}
				\frac{K_1^{\prime}}{4}(|F-\bar{F}|^p+|\theta-\bar{\theta}|^{\ell}+|v-\bar{v}|^2),\quad\quad\quad |F|^p\!+\theta^{\ell}\!+|v|^2>R\\
				K_2^{\prime}(|\Phi(F)-\Phi(\bar{F})|^2\!+\!|\theta-\bar{\theta}|^2\!+\!|v-\bar{v}|^2), \quad |F|^p\!+\theta^{\ell}\!+|v|^2\leq R
				\end{dcases}
				\end{split}
				\end{align}
			\end{small}
			for all $(\bar{F},\bar{v},\bar{\theta})\in\Gamma_{M,\delta}$.
		\end{enumerate}
	\end{lemma}
	\begin{proof}
		%Let us first define the sets:
		%\begin{align*}
		%&\hat{\Gamma}_{M,\delta}:=\left\{(\bar{F},\bar{\theta}):: |\bar{F}|\leq M, 0<\delta\leq\bar{\theta}\leq M\right\},\\
		%&\hat{D}:=\left\{(\Phi(F),\theta): |F|\leq R^{1/p}, |\mathrm{cof}F|\leq CR^{2/p}, |\det F|\leq CR^{3/p}, 0<\theta\leq R^{1/\ell}\right\}.
		%\end{align*}
		We divide the proof into 5 steps.
		
		\textbf{Step $1.$} To prove (\ref{bound6}), we use (\ref{Sigma}) to obtain
		\begin{align*}
		&\left|\left(\frac{\partial\Phi^B}{\partial F_{i\alpha}}(F)
		-\frac{\partial\Phi^B}{\partial F_{i\alpha}}(\bar{F})\right)
		\left(\frac{\partial\hat{\psi}}{\partial\xi^B}(\Phi(F),\theta)
		-\frac{\partial\hat{\psi}}{\partial\xi^B}(\Phi(\bar{F}),\bar{\theta})\right)\right|\\
		&=\frac{\partial\hat{\psi}}{\partial\xi^B}(\Phi(F),\theta)\frac{\partial\Phi^B}{\partial F_{i\alpha}}(F)
		-\frac{\partial\hat{\psi}}{\partial\xi^B}(\Phi(\bar{F}),\bar{\theta})\frac{\partial\Phi^B}{\partial F_{i\alpha}}(F)
		-\frac{\partial\hat{\psi}}{\partial\xi^B}(\Phi(F),\theta)\frac{\partial\Phi^B}{\partial F_{i\alpha}}(\bar{F})\\
		&\qquad+\frac{\partial\hat{\psi}}{\partial\xi^B}(\Phi(\bar{F}),\bar{\theta})\frac{\partial\Phi^B}{\partial F_{i\alpha}}(\bar{F})\\
		&=\Sigma(F,\theta)-\frac{\partial\hat{\psi}}{\partial\xi^B}(\Phi(\bar{F}),\bar{\theta})\frac{\partial\Phi^B}{\partial F_{i\alpha}}(F)
		-\frac{\partial\hat{\psi}}{\partial\xi^B}(\Phi(F),\theta)\frac{\partial\Phi^B}{\partial F_{i\alpha}}(\bar{F})+\Sigma(\bar{F},\bar{\theta}).
		\end{align*}
		For $|F|^p+\theta^{\ell}+|v|^2>R$ and $(\bar{F},\bar{v},\bar{\theta})\in \Gamma_{M,\delta},$
		using (\ref{gr.con.4}), (\ref{gr.con.5}), (\ref{gr.con.2}) and Young's inequality we have
		\begin{align*}
		&\left|\left(\frac{\partial\Phi^B}{\partial F_{i\alpha}}(F)
		-\frac{\partial\Phi^B}{\partial F_{i\alpha}}(\bar{F})\right)
		\left(\frac{\partial\hat{\psi}}{\partial\xi^B}(\Phi(F),\theta)
		-\frac{\partial\hat{\psi}}{\partial\xi^B}(\Phi(\bar{F}),\bar{\theta})\right)\right|\\
		&\qquad\qquad\qquad\leq |\Sigma(F,\theta)|+c_1\left|\frac{\partial\Phi^B}{\partial F}(F)\right|+c_2\left|\frac{\partial\hat{\psi}}{\partial\xi^B}(\Phi(F),\theta)\right|+c_3\\
		&\qquad\qquad\qquad\leq c_4|\psi(F,\theta)|+c_5(|F|^2+|F|+1)+c_6|\hat{\psi}(\Phi(F),\theta)|+c_3\\
		&\qquad\qquad\qquad\leq c_7(|F|^p+\theta^{\ell})+c_8.
		\end{align*}
		Selecting now $R$ large enough, so that $c_8<c(|F|^p\!+\theta^{\ell}\!+|v|^2)$ for $p>3,\:\ell>1,$ we conclude that
		\begin{align*}
		\left|\left(\frac{\partial\Phi^B}{\partial F_{i\alpha}}(F)
		-\frac{\partial\Phi^B}{\partial F_{i\alpha}}(\bar{F})\right)
		\left(\frac{\partial\hat{\psi}}{\partial\xi^B}(\Phi(F),\theta)
		-\frac{\partial\hat{\psi}}{\partial\xi^B}(\Phi(\bar{F}),\bar{\theta})\right)\right|
		\leq C_1' I(\Phi(F),v,\theta|\Phi(\bar{F}),\bar{v},\bar{\theta})
		\end{align*}
		by Lemma \ref{lemma1}. In the region $|F|^p+|v|^2+\theta^{\ell}\leq R$, we have that $(\Phi(F),v,\theta),
		(\Phi(\bar{F}),\bar{v},\bar{\theta})\in D$ so that 
		\begin{align*}
		&\left|\left(\frac{\partial\Phi^B}{\partial F_{i\alpha}}(F)
		-\frac{\partial\Phi^B}{\partial F_{i\alpha}}(\bar{F})\right)
		\left(\frac{\partial\hat{\psi}}{\partial\xi^B}(\Phi(F),\theta)
		-\frac{\partial\hat{\psi}}{\partial\xi^B}(\Phi(\bar{F}),\bar{\theta})\right)\right|\\
		&\qquad\qquad\qquad\leq \max_{D}\left|\nabla^2_{\xi}  
		\hat\psi(\xi,\theta)
		\right| \left|\Phi(F)-\Phi(\bar{F})\right| \left|\frac{\partial\Phi}{\partial F}(F)-\frac{\partial\Phi}{\partial F}(\bar{F})\right|
		\\
		%		%%%
		%		&{\color{red}
		%			\qquad\qquad\qquad\leq \max_{D}\left|\nabla^2_{\xi}\:\;\left(\pg\right)\right| \left|\Phi(F)-\Phi(\bar{F})\right| \left|\frac{\partial\Phi}{\partial F}(F)-\frac{\partial\Phi}{\partial F}(\bar{F})\right|
		%		}\\
		&\qquad\qquad\qquad\leq c_1 \left|\Phi(F)-\Phi(\bar{F})\right|^2\\
		&\qquad\qquad\qquad\leq C_1'' I(\Phi(F),v,\theta|\Phi(\bar{F}),\bar{v},\bar{\theta})
		\end{align*}
		using again Lemma \ref{lemma1}. Choosing now $C_1=\max\{C_1',C_1''\},$ estimate (\ref{bound6}) follows.
		
		\textbf{Step $2.$} Using Young's inequality and (\ref{gr.con.5}), it follows
		\begin{align*}
		\left|\frac{\partial \hat{\psi}}{\partial \xi}(\Phi(F),\theta|\Phi(\bar{F}),\bar{\theta})\right|
		&\leq \left|\frac{\partial \hat{\psi}}{\partial \xi}(\Phi(F),\theta)\right|+c_1+c_2|\Phi(F)|+c_3|\theta|\\
		&\leq c_4 |\hat{\psi}(\Phi(F),\theta)|+c_2(|F|^3+|F|^2+|F|)+c_3|\theta|+c_1\\
		&\leq c_5(|F|^p+\theta^{\ell})+c_6\;.
		\end{align*}
		Choosing again $R$ large enough, such that $R>r(M)+1$ there holds
		\begin{align*}
		\left|\frac{\partial \hat{\psi}}{\partial \xi}(\Phi(F),\theta|\Phi(\bar{F}),\bar{\theta})\right|
		&\leq c_7(|F|^p+\theta^{\ell}+|v|^2)\\
		&\leq C_2' I(\Phi(F),v,\theta|\Phi(\bar{F}),\bar{v},\bar{\theta})
		\end{align*}
		by (\ref{bound4}) and for $|F|^p+\theta^{\ell}+|v|^2>R$ and $(\bar{F},\bar{v},\bar{\theta})\in\Gamma_{M,\delta}.$
		In the complementary region $|F|^p+\theta^{\ell}+|v|^2\leq R,$ there holds $(\Phi(F),v,\theta),(\Phi(\bar{F}),\bar{v},\bar{\theta})\in D$, therefore
		\begin{align*}
		\left|\frac{\partial \hat{\psi}}{\partial \xi}(\Phi(F),\theta|\Phi(\bar{F}),\bar{\theta})\right|
		&\leq \max_{D}\left|\nabla_{(\xi,\theta)}^2\hat{\psi}(\xi,\theta)\right| (|\Phi(F)-\Phi(\bar{F})|^2+|\theta-\bar{\theta}|^2)\\
		&\leq C_2'' I(\Phi(F),v,\theta|\Phi(\bar{F}),\bar{v},\bar{\theta})
		\end{align*}
		again by (\ref{bound4}). Choosing $C_2=\max\{C_2',C_2''\},$ the proof of (\ref{bound7}) is complete.
		
		\textbf{Step $3.$} We proceed in a similar manner as in Step $2.$ to prove (\ref{bound5}).
		First we study the region $|F|^p+|v|^2+\theta^{\ell} >R$ and we use growth assumption (\ref{gr.con.3}) and relation (\ref{eta&e}) to get
		\begin{small}
			\begin{align*}
			&\lim_{|F|^p+\theta^{\ell}\to\infty}\frac{|\hat{\eta}(\Phi(F),\theta|\Phi(\bar{F}),\bar{\theta})|}{|F|^p+\theta^{\ell}}
			=\lim_{|F|^p+\theta^{\ell}\to\infty}\frac{|\hat{\eta}(\Phi(F),\theta)|}{|F|^p+\theta^{\ell}}
			=\lim_{|F|^p+\theta^{\ell}\to\infty}\frac{|\eta(F,\theta)|}{|F|^p+\theta^{\ell}}=0.
			\end{align*}
		\end{small}
		So immediately we deduce $$|\hat{\eta}(\Phi(F),\theta|\Phi(\bar{F}),\bar{\theta})|\leq C_3' I(\Phi(F),v,\theta|\Phi(\bar{F}),\bar{v},\bar{\theta}),$$
		for $R$ large enough. On the complementary region $|F|^p+|v|^2+\theta^{\ell}\leq R,$ 
		\begin{align*}
		|\hat{\eta}(\Phi(F),\theta|\Phi(\bar{F}),\bar{\theta})|
		&\leq\max_{D}|\nabla_{(\xi,\theta)}^2\hat{\eta}(\Phi(F),\theta)|(|\Phi(F)-\Phi(\bar{F})|^2+|\theta-\bar{\theta}|^2)\\
		&\leq C_3'' I(\Phi(F),v,\theta|\Phi(\bar{F}),\bar{v},\bar{\theta})\;,
		\end{align*}
		by~\eqref{bound4}. Choosing $C_3=\max\{C_3',C_3''\},$ the proof of (\ref{bound5}) is complete.
		
		\textbf{Step $4.$} Similarly, for (\ref{bound8}), when $|F|^p+\theta^{\ell} +|v|^2>R$ and $(\bar{F},\bar{v},\bar{\theta})\in\Gamma_{M,\delta}$, we have
		\begin{align*}
		\left|\left(\ph-\phb\right)(v_i-\bar{v}_i)\right|
		\leq c_1|\Phi(F)|^2+ c_2|v|^2 +c_3
		\end{align*}
		and choosing appropriately the radius $R,$ proceeding as before, it follows 
		\begin{align*}
		\left|\left(\ph-\phb\right)(v_i-\bar{v}_i)\right|
		&\leq c_4(|F|^p+\theta^{\ell}+|v|^2)\\
		&\leq C_4' I(\Phi(F),v,\theta|\Phi(\bar{F}),\bar{v},\bar{\theta}),
		\end{align*}
		where we use again~\eqref{bound4} and $p>3.$
		Then, for $|F|^p+\theta^{\ell}+|v|^2\leq R$ and for all $(\Phi(F),v,\theta),$ $(\Phi(\bar{F}),\bar{v},\bar{\theta})\in D$, we also get
		\begin{align*}
		\left|\left(\ph-\phb\right)(v_i-\bar{v}_i)\right|&\leq\
		\tfrac{1}{2} \left| \frac{\partial\Phi^B}{\partial F}(F)-\frac{\partial\Phi^B}{\partial F}(\bar{F}) \right|^2+ \tfrac{1}{2} |v-\bar{v}|^2\\
		&\leq c_1 (|\Phi(F)-\Phi(\bar{F})|^2+|v-\bar{v}|^2)\\
		&\leq C_4'' I(\Phi(F),v,\theta|\Phi(\bar{F}),\bar{v},\bar{\theta})
		\end{align*} 
		by~\eqref{bound4}. Choosing $C_4=\max\{C_4',C_4''\},$ estimate (\ref{bound8}) follows.
		
		\textbf{Step $5.$}  Since $(\bar{F},\bar{v},\bar{\theta})\in\Gamma_{M,\delta}\subset B_r$ -for sufficiently large $R$- there holds
		\begin{align*}
		|F-\bar{F}|^p&+|\theta-\bar{\theta}|^{\ell}+|v-\bar{v}|^2\leq(|F|+M)^p+(\theta+M)^{\ell}+(|v|+M)^2\;
		\end{align*}
		and
		\begin{align*}
		\lim_{|F|^p+\theta^{\ell}+|v|^2\to\infty}\frac{(|F|+M)^p+(\theta+M)^{\ell}+(|v|+M)^2}{|F|^p+\theta^{\ell}+|v|^2}=1\;.
		\end{align*}
		Thus, we may select $R$ such that
		\begin{align*}
		|F-\bar{F}|^p+|\theta-\bar{\theta}|^{\ell}+|v-\bar{v}|^2&\leq2(|F|^p+\theta^{\ell}+|v|^2+1)\\
		&\leq C(|F|^p+\theta^{\ell}+|v|^2)
		\end{align*}
		when $|F|^p+\theta^{\ell}+|v|^2 \geq R.$ Thus~(\ref{bound9}) follows from~\eqref{bound4}.
		This concludes the proof.
	\end{proof}
	
	We now consider a dissipative measure-valued solution for polyconvex thermoelasticity as defined in 
	Definition \ref{def.mv.F}. Using the averaged relative entropy inequality (\ref{rel.en.id.mv}), we prove that in the 
	presence of a classical solution, given that the associated Young measure is initially a Dirac mass, the \emph{dissipative measure-valued} solution
	must coincide with the \emph{classical} one.
	\begin{theorem}
		\label{Uniqueness_thm.mv}
		Let $\bar{U}$ be a Lipschitz bounded solution of (\ref{sys2}),(\ref{entropy.prod}) with initial data $\bar{U}^0$ and $(\bn,\bg,U)$ be a dissipative measure-valued solution  satisfying (\ref{mv.sol.Fvt}),(\ref{mv.sol.Fvt.energy}), with initial data $U^0,$ both under the constitutive assumptions (\ref{con.rel.aug}) and such that $r(x,t)=\bar{r}(x,t)=0$.
		Suppose that $\nabla_{\xi}^2\hat{\psi}(\Phi(F),\theta)>0$ and $\hat{\eta}_{\theta}(\Phi(F),\theta)>0$ and the 
		growth conditions (\ref{gr.con.1}), (\ref{gr.con.2}), (\ref{gr.con.3}), (\ref{gr.con.4}), (\ref{gr.con.5})
		hold for $p\geq 4,$ and $\ell> 1.$
		If $\bar{U}\in\Gamma_{M,\delta},$ for some positive constants $M,\delta$ and $\bar{U}\in W^{1,\infty}(Q_T)$, whenever $\bn_{(0,x)}=\delta_{\bar{U}^0}(x)$ and $\bg_0=0$ %{\color{red} $\bn_0=0,$} 
		we have that $\bn=\delta_{\bar{U}}$ and $U=\bar U$ a.e. on $Q_T.$
	\end{theorem}
	\begin{proof} 
		Let $\{\varphi_n\}$ be a sequence of monotone decreasing functions such that $\varphi_n\geq 0,$ for all $n\in\mathbb{N},$ converging as $n \to \infty$ to the Lipschitz function
		\begin{align*}
		\varphi(\tau)=\begin{dcases}
		1 & 0\leq\tau\leq t\\
		\frac{t-\tau}{\varepsilon}+1 & t\leq\tau\leq t+\varepsilon\\
		0 & \tau\geq t+\varepsilon
		\end{dcases}
		\end{align*}
		for some $\varepsilon>0.$ Writing the relative entropy inequality~\eqref{rel.en.id.mv} for $r(x,t)=\bar{r}(x,t)=0,$ tested against the functions $\varphi_n$ we have 
		\begin{small}
			\begin{align}
			\label{rel.entr.wk-n.mv}
			\begin{split}
			\int&\varphi_n(0)\left\la\bn,I(\lambda_{U_0}|\bar{U_0})\right\ra\:dx
			+\int_{0}^{t}\int\varphi'_n(\tau)\left[\left\la\bn,I(\lambda_{U}|\bar{U})\right\ra\:dx\:d\tau+\bg(dx\,d\tau)\right]\\
			&\geq-\int_{0}^{t}\int\varphi_n(\tau)\left[-\partial_t\bar{\theta}\left\la\bn,\hat{\eta}(\Phi(\lf),\lt|\Phi(\bar{F}),\bar{\theta})\right\ra
			+\partial_t\Phi^B(\bar{F})\left\la\bn,\frac{\partial\hat{\psi}}{\partial \xi^B}(\Phi(\lf),\lt|\Phi(\bar{F}),\bar{\theta})\right\ra\right.\\
			%&\!\!\!\!\!\!\!\!\!\!
			&\quad
			\left.+\partial_{\alpha}\bar{v}_i\left\la\bn,\left(\frac{\partial\hat{\psi}}{\partial\xi^B}(\Phi(\lambda_{F}),\lambda_{\theta})
			-\frac{\partial\hat{\psi}}{\partial\xi^B}(\Phi(\bar{F}),\bar{\theta})\right)\left(\phl-\phb\right)\right\ra\right.\\
			%&\!\!\!\!\!\!\!\!\!\!
			&\quad
			\left.+\partial_{\alpha}\left(\frac{\partial\hat{\psi}}{\partial\xi^B}(\Phi(\bar{F}),\bar{\theta})\right)\!
			\left(\ph-\phb\right)\!\la\bn,(\lvi-\bar{v}_i)\ra\right]\:dx\,d\tau.
			\end{split}
			\end{align}
			%		\left[-\partial_t\bar{\theta}\left\la\bn,\hat{\eta}(\Phi(\lf),\lt|\Phi(\bar{F}),\bar{\theta})\right\ra
			%		+\partial_t\bar{\Phi}^B(\bar{F})\left\la\bn,\frac{\partial\hat{\psi}}{\partial \xi^B}(\Phi(\lf),\lt|\Phi(\bar{F}),\bar{\theta})\right\ra\right.\\
			%		&\!\!\!\!\!\!\!\!\!\!\left.+\partial_{\alpha}\bar{v}_i\left\la\bn,\left(\frac{\partial\hat{\psi}}{\partial\xi^B}(\Phi(\lambda_{F}),\lambda_{\theta})
			%		-\frac{\partial\hat{\psi}}{\partial\xi^B}(\Phi(\bar{F}),\bar{\theta})\right)\left(\phl-\phb\right)\right\ra\right.\\
			%		&\!\!\!\!\!\!\!\!\!\!\left.+\partial_{\alpha}\left(\frac{\partial\hat{\psi}}{\partial\xi^B}(\Phi(\bar{F}),\bar{\theta})\right)\!\left(\ph-\phb\right)\!\la\bn,(\lv-\bar{v})_i\ra
			%		+\left\la\bn,\left(\frac{r}{\lt}-\frac{\bar{r}}{\bar{\theta}}\right)(\lt-\bar{\theta})\right\ra\right]\:dx\,dt.
		\end{small} 
		Passing to the limit as $n\to \infty$ we get
		\begin{small}
			\begin{align*}
			\int &\left\la\bn,I(\Phi(\lf),\lv,\lt| \Phi(\bar{F}),\bar{v},\bar{\theta})\right\ra(x,0)\:dx\\
			&\quad
			-\frac{1}{\varepsilon}\int_{t}^{t+\varepsilon}
			\int \left[\left\la\bn,I(\Phi(\lf),\lv,\lt| \Phi(\bar{F}),\bar{v},\bar{\theta})\right\ra\:dx\:d\tau+
			\bg(dxd\tau)\right]  \\ 
			&\geq- \int_0^{t+\varepsilon}\!\!\!\int \!\left[-\partial_t\bar{\theta}\left\la\bn,\hat{\eta}(\Phi(\lf),\lt|\Phi(\bar{F}),\bar{\theta})\right\ra
			+\partial_t\Phi^B(\bar{F})\left\la\bn,\frac{\partial\hat{\psi}}{\partial \xi^B}(\Phi(\lf),\lt|\Phi(\bar{F}),\bar{\theta})\right\ra\right.\\
			&\quad
			\left.+\partial_{\alpha}\bar{v}_i\left\la\bn,\left(\frac{\partial\hat{\psi}}{\partial\xi^B}(\Phi(\lambda_{F}),\lambda_{\theta})
			-\frac{\partial\hat{\psi}}{\partial\xi^B}(\Phi(\bar{F}),\bar{\theta})\right)\left(\phl-\phb\right)\right\ra\right.\\
			&\quad
			\left.+\partial_{\alpha}\left(\frac{\partial\hat{\psi}}{\partial\xi^B}(\Phi(\bar{F}),\bar{\theta})\right)\!
			\left(\ph-\phb\right)\!\la\bn,(\lvi-\bar{v}_i)\ra\right]\:dx\,d\tau.
			\end{align*}
		\end{small}
		Passing now to the limit as $\varepsilon \to 0^{+}$ and using the fact that $\bg\geq 0$ in combination with
		the estimates (\ref{bound6}), (\ref{bound7}),~(\ref{bound5}) and~\eqref{bound8}, we arrive at
		\begin{align*}
		\int  \left\la\bn,I(\Phi(\lf),\lv,\lt| \Phi(\bar{F}),\bar{v},\bar{\theta})\right\ra\:dx\:dt 
		\leq C  &\int_0^t\int \left\la\bn,I(\Phi(\lf),\lv,\lt| \Phi(\bar{F}) ,\bar{v},\bar{\theta})\right\ra\:dx\:d\tau\\
		&+\int \left\la\bn,I(\Phi(\lf),\lv,\lt| \Phi(\bar{F}),\bar{v},\bar{\theta})\right\ra(x,0)\:dx
		\end{align*}
		for $t\in (0,T).$ Note that the constant $C$ depends only on the smooth bounded solution $\bar{U}.$ 
		Then Gronwall's inequality implies 
		\begin{align*} 
		\int  \left\la\bn,I(\Phi(\lf),\lv,\lt| \Phi(\bar{F}),\bar{v},\bar{\theta})\right\ra\:dx\:dt 
		\leq C_1 e^{C_2t}\int \left\la\bn,I(\Phi(\lf),\lv,\lt| \Phi(\bar{F})\bar{v},\bar{\theta})\right\ra(x,0)\:dx
		\end{align*}
		and the proof is complete by (\ref{bound9}).
	\end{proof}
	
	%\begin{remark}\rm
	An extension of Theorem \ref{Uniqueness_thm.mv} holds in case we assume $r(x,t)=\bar{r}(x,t)\neq 0.$
	For this purpose, we need the additional assumption
	\begin{align}
	\label{supp.nu}
	\mathrm{supp}\:\bn\subset\:\mathbb{R}^{19}\times\mathbb{R}^3\times[\underline{\delta},\infty).
	\end{align} 
	to control the terms that arise from the radiative heat supply in~\eqref{rel.en.id.mv}.
	We first prove the following lemma:
	\begin{lemma}
		\label{lemma3}
		Suppose that $r(x,t)=\bar{r}(x,t)\in L^{\infty}(Q_T)$ and that
		\begin{align*}
		%\label{supp.nu}
		\mathrm{supp}\:\bn\subset\:\mathbb{R}^{19}\times\mathbb{R}^3\times[\underline{\delta},\infty),
		\end{align*}
		for some small positive constant $\underline{\delta}.$ Then there exists a constant $C_5>0$ such that
		\begin{equation}
		\label{bound10}
		\left|\left\la\bn,\left(\frac{r}{\lt}-\frac{\bar{r}}{\bar{\theta}}\right)(\lt-\bar{\theta})\right\ra\right|
		\leq C_5 \left\la\bn,I(\Phi(F),v,\theta|\Phi(\bar{F}),\bar{v},\bar{\theta})\right\ra
		\end{equation}
		for all $(\bar{F},\bar{v},\bar{\theta})\in\Gamma_{M,\delta}$.
	\end{lemma}
	\begin{proof} 
		Assume first that $|F|^p +\theta^{\ell}+|v|^2>R.$ Then
		\begin{align*}
		\left|\left\la\bn,\left(\frac{r}{\lt}-\frac{\bar{r}}{\bar{\theta}}\right)(\lt-\bar{\theta})\right\ra\right|
		&=\left|\int\left(\frac{\bar{r}}{\lt}-\frac{\bar{r}}{\bar{\theta}}\right)(\lt-\bar{\theta})\:d\bn\right|\\
		&\leq \|\bar{r}\|_{L^{\infty}}\left|\int \frac{(\lt-\bar{\theta})^2}{\lt\bar{\theta}}\:d\bn\right|\\
		&=\|\bar{r}\|_{L^{\infty}} \left|\int\left(\frac{\lt}{\bar{\theta}}-2+\frac{\bar{\theta}}{\lt}\right)\:d\bn\right|\\
		&\leq \|\bar{r}\|_{L^{\infty}} \left(\left|\left\la\bn,\frac{\lt}{\bar{\theta}}\right\ra\right|+\left|\left\la\bn,\frac{\bar{\theta}}{\lt}\right\ra\right|+c_1\right)\\
		&\leq c_2\left|\left\la\bn,\lt\right\ra\right|+C_3\left|\left\la\bn,1\right\ra\right|+c_4\\
		&\leq c_5 (|\theta|+1)\;.
		\end{align*}
		Choosing $R$ sufficiently large, we get for $\ell>1$
		\begin{align*}
		\left|\left\la\bn,\left(\frac{r}{\lt}-\frac{\bar{r}}{\bar{\theta}}\right)(\lt-\bar{\theta})\right\ra\right|
		&\leq c_6 (|F|^p+|\theta|^{\ell}+|v|^2)\\
		&\leq C_5'\left\la\bn,I(\Phi(F),v,\theta|\Phi(\bar{F}),\bar{v},\bar{\theta})\right\ra,
		\end{align*}
		where, the last inequality holds because of Lemma \ref{lemma1} and the constant $C_5'$ depends on $\bar{r},\:\delta,\:M$ and $\underline{\delta}.$
		
		Now, similarly, if $|F|^p +\theta^{\ell}+|v|^2\leq R,$ we have
		\begin{align*}
		\left|\left\la\bn,\left(\frac{r}{\lt}-\frac{\bar{r}}{\bar{\theta}}\right)(\lt-\bar{\theta})\right\ra\right|&\leq \|\bar{r}\|_{L^{\infty}}
		\int \frac{|\lt-\bar{\theta}|^2}{|\lt\bar{\theta}|}\:d\bn\\
		&\leq C_1 \int |\lt-\bar{\theta}|^2\:d\bn\\
		&\leq C_5''\left\la\bn,I(\Phi(F),v,\theta|\Phi(\bar{F}),\bar{v},\bar{\theta})\right\ra\;,
		\end{align*}
		again by estimate (\ref{bound4}). By choosing $$C_5=\max\{C_5'(\bar{r},\delta,M,\underline{\delta}),C_5''(\bar{r},\delta,\underline{\delta})\},$$
		the proof is complete.
	\end{proof}
	
	Then Theorem \ref{Uniqueness_thm.mv} extends to :
	\begin{theorem}
		\label{Uniqueness_thm.mv.with.r}
		Let $\bar{U}$ be a Lipschitz bounded solution of (\ref{sys2}),(\ref{entropy.prod}) with initial data $\bar{U}^0$ and $(\bn,\bg,U)$ be a dissipative measure-valued solution  satisfying (\ref{mv.sol.Fvt}),(\ref{mv.sol.Fvt.energy}), with initial data $U^0,$ both under the constitutive assumptions (\ref{con.rel.aug}) and such that $r(x,t)=\bar{r}(x,t)$. Assume also that there exists a small constant $\underline{\delta}>0$ such that (\ref{supp.nu}) holds true. Suppose that $\nabla_{\xi}^2\hat{\psi}(\Phi(F),\theta)>0$ and $\hat{\eta}_{\theta}(\Phi(F),\theta)>0$ and the growth
		conditions (\ref{gr.con.1}), (\ref{gr.con.2}), (\ref{gr.con.3}), (\ref{gr.con.4}), (\ref{gr.con.5})
		hold for $p\geq 4,$ and $\ell> 1.$
		If $\bar{U}\in\Gamma_{M,\delta},$ for some positive constants $M,\delta$ and $\bar{U}\in W^{1,\infty}(Q_T)$, whenever 
		$\bn_{(0,x)}=\delta_{\bar{U}^0}(x),$ and  $\bg_0=0$ %{\color{red} $\bn_0=0,$} 
		we have that $\bn=\delta_{\bar{U}}$ and $U=\bar U$ a.e. on $Q_T.$
	\end{theorem}
	\begin{proof} 
		The proof is a simple variant of the one for Theorem \ref{Uniqueness_thm.mv}. Assuming the sequence
		$\{\varphi_n\}$ as before, the relative entropy inequality~\eqref{rel.en.id.mv} becomes
		\begin{small}
			\begin{align*}
			\begin{split}
			\int&\varphi_n(0)\left\la\bn,I(\lambda_{U_0}|\bar{U_0})\right\ra\:dx
			+\int_{0}^{t}\int\varphi'_n(\tau)\left[\left\la\bn,I(\lambda_{U}|\bar{U})\right\ra\:dx\:d\tau+\bg(dx\,d\tau)\right]\\
			&\geq-\int_{0}^{t}\int\varphi_n(\tau)\left[-\partial_t\bar{\theta}\left\la\bn,\hat{\eta}(\Phi(\lf),\lt|\Phi(\bar{F}),\bar{\theta})\right\ra
			+\partial_t\Phi^B(\bar{F})\left\la\bn,\frac{\partial\hat{\psi}}{\partial \xi^B}(\Phi(\lf),\lt|\Phi(\bar{F}),\bar{\theta})\right\ra\right.\\
			&\!\!\!\!\!\!\!\!\!\!\left.+\partial_{\alpha}\bar{v}_i\left\la\bn,\left(\frac{\partial\hat{\psi}}{\partial\xi^B}(\Phi(\lambda_{F}),\lambda_{\theta})
			-\frac{\partial\hat{\psi}}{\partial\xi^B}(\Phi(\bar{F}),\bar{\theta})\right)\left(\phl-\phb\right)\right\ra\right.\\
			&\!\!\!\!\!\!\!\!\!\!\left.+\partial_{\alpha}\left(\frac{\partial\hat{\psi}}{\partial\xi^B}(\Phi(\bar{F}),\bar{\theta})\right)\!
			\left(\ph-\phb\right)\!\la\bn,(\lvi-\bar{v}_i)\ra
			+\left\la\bn,\left(\frac{r}{\lt}-\frac{\bar{r}}{\bar{\theta}}\right)(\lt-\bar{\theta})\right\ra\right]\:dx\,d\tau.
			\end{split}
			\end{align*}
		\end{small} 
		Passing to the limit as $n\to \infty$ and then as $\varepsilon \to 0^{+}$ we obtain
		\begin{align*}
		\int  \left\la\bn,I(\Phi(\lf),\lv,\lt|\Phi(\bar{F}),\bar{v},\bar{\theta})\right\ra\:dx\:dt 
		\leq C  &\int_0^t\int \left\la\bn,I(\Phi(\lf),\lv,\lt|\Phi(\bar{F}),\bar{v},\bar{\theta})\right\ra\:dx\:d\tau\\
		&+\int \left\la\bn,I(\Phi(\lf),\lv,\lt| \Phi(\bar{F}),\bar{v},\bar{\theta})\right\ra(x,0)\:dx
		\end{align*}
		for $t\in (0,T).$ Here, we used that $\bg\ge 0$ and the estimates (\ref{bound6}), (\ref{bound7}), (\ref{bound5}), (\ref{bound8}) and (\ref{bound10}), so that constant $C$ depends on the smooth bounded solution $\bar{U}$ and $\underline{\delta}.$ 
		%It follows
		%\begin{align*}
		%\int  \left\la\bn,I(\Phi(\lf),\lv,\lt|  \Phi(\bar{F}),\bar{v},\bar{\theta})\right\ra\:dx\:dt 
		%\leq C_1e^{C_2t}\int \left\la\bn,I(\Phi(\lf),\lv,\lt| \Phi(\bar{F}),\bar{v},\bar{\theta})\right\ra(x,0)\:dx
		%\end{align*}
		By virtue of Gronwall's inequality and~\eqref{bound9}, we conclude the proof.
		%    	\begin{small}
		%    		\begin{align*}
		%    		\int &\left\la\bn,I(\Phi(\lf),\lv,\lt|\bar{\xi},\bar{v},\bar{\theta})\right\ra(x,0)\:dx
		%    		-\frac{1}{\varepsilon}\int_{t}^{t+\varepsilon}
		%    		\int \left\la\bn,I(\Phi(\lf),\lv,\lt|\bar{\xi},\bar{v},\bar{\theta})\right\ra\:dx\:d\tau  \\ 
		%    		&\geq- \int_0^{t+\varepsilon}\!\!\!\int \!\left[-\partial_t\bar{\theta}\left\la\bn,\hat{\eta}(\Phi(\lf),\lt|\Phi(\bar{F}),\bar{\theta})\right\ra
		%    		+\partial_t\Phi^B(\bar{F})\left\la\bn,\frac{\partial\hat{\psi}}{\partial \xi^B}(\Phi(\lf),\lt|\Phi(\bar{F}),\bar{\theta})\right\ra\right.\\
		%    		&\!\!\!\!\!\!\!\!\!\!\left.+\partial_{\alpha}\bar{v}_i\left\la\bn,\left(\frac{\partial\hat{\psi}}{\partial\xi^B}(\Phi(\lambda_{F}),\lambda_{\theta})
		%    		-\frac{\partial\hat{\psi}}{\partial\xi^B}(\Phi(\bar{F}),\bar{\theta})\right)\left(\phl-\phb\right)\right\ra\right.\\
		%    		&\!\!\!\!\!\!\!\!\!\!\left.+\partial_{\alpha}\left(\frac{\partial\hat{\psi}}{\partial\xi^B}(\Phi(\bar{F}),\bar{\theta})\right)\!
		%    		\left(\ph-\phb\right)\!\la\bn,(\lv-\bar{v})_i\ra
		%    		+\left\la\bn,\left(\frac{r}{\lt}-\frac{\bar{r}}{\bar{\theta}}\right)(\lt-\bar{\theta})\right\ra\right]\:dx\,d\tau.
		%    		\end{align*}
		%    	\end{small}
	\end{proof}
	
	Let us note that,
	as Lemma \ref{lemma3} indicates, one needs to assume (\ref{supp.nu}) in order to be able to bound from below the averaged temperature $\left\la\bn,\lt\right\ra$ and
	achieve estimate (\ref{bound10}). Though it could be considered as a rather mild assumption, it is interesting that all the estimates in Lemmas \ref{lemma1} and
	\ref{lemma2} that involve the averaged temperature, do not require (\ref{supp.nu}) to hold. This is because the averaged temperature is involved only through the
	constitutive functions $\hat{\psi},\:\hat{e}$ and $\hat{\eta}$ which we assume to be smooth enough, i.e. $\hat{\psi}=\hat{e}-\theta\hat{\eta}\in C^2$, and therefore we avoid
	any loss of smoothness as the temperature approaches zero.
	%\end{remark}
	
	\appendix
	\section{The natural bounds of viscous approximation for polyconvex thermoelasticity} \label{AppA}   %Viscosity approximation mv solutions
	Since measure-valued solutions usually occur as limits of an approximating problem, consider the system of polyconvex thermoelasticity with Newtonian viscosity and Fourier heat conduction
	\begin{small}
		\begin{align}
		\begin{split}
		\label{poly.sys.mu.k}
		\partial_t \Phi^B(F^{\mu,k})&=\partial_{\alpha}\left(\frac{\partial\Phi^B}{\partial F_{i\alpha}}(F^{\mu,k})v_i^{\mu,k}\right)\\ 
		\partial_t v^{\mu,k}_i
		&=\partial_{\alpha}\left(\frac{\partial\hat{\psi}}{\partial\xi^B}(\Phi(F^{\mu,k}),\theta^{\mu,k})
		\frac{\partial\Phi^B}{F_{i\alpha}}(F^{\mu,k})\right) 
		+\partial_{\alpha}(\mu\partial_{\alpha}v_i^{\mu,k})\\
		\partial_t \left(\frac{1}{2}|v^{\mu,k}|^2+\hat{e}(\Phi(F^{\mu,k}),\theta^{\mu,k})\right)
		&=\partial_{\alpha}\left(\frac{\partial\hat{\psi}}{\partial\xi^B}(\Phi(F^{\mu,k}),\theta^{\mu,k})
		\frac{\partial\Phi^B}{F_{i\alpha}}(F^{\mu,k}) v_i^{\mu,k}\right)+\\
		&\quad\quad\quad\quad\quad\quad\quad\quad\quad\quad
		+\partial_{\alpha}(\mu v_i^{\mu,k}\partial_{\alpha}v_i^{\mu,k}+k\partial_{\alpha}\theta^{\mu,k})+r\\
		\partial_t\hat{\eta}(\Phi(F^{\mu,k}),\theta^{\mu,k})&=\partial_{\alpha}\left(k\frac{\nabla\theta^{\mu,k}}{\theta^{\mu,k}}\right)
		+k\frac{|\nabla\theta^{\mu,k}|^2}{(\theta^{\mu,k})^2}+\mu\frac{|\nabla v^{\mu,k}|^2}{\theta^{\mu,k}}+\frac{r}{\theta^{\mu,k}}
		\end{split}
		\end{align}
	\end{small}
	where
	\begin{equation*}
	\frac{\partial \hat{\psi}}{\partial F_{i\alpha}}(\Phi(F),\theta)=\frac{\partial\psi}{\partial F_{i\alpha}}(\Phi(F),\theta)
	\frac{\partial\psi}{\partial \zeta^{k\gamma}}(\Phi(F),\theta)
	\frac{\partial(\mathrm{cof}F)_{k\gamma}}{\partial F_{i\alpha}}
	+\frac{\partial\psi}{\partial w}(\Phi(F),\theta)\frac{\partial(\det F)}{\partial F_{i\alpha}}.
	\end{equation*}
	%Given that the energy radiation $r$ is a bounded function 
	Suppose first that the energy radiation $r\equiv 0.$  The viscosity and heat conduction coefficients are assumed to satisfy
	the condition
	\begin{align}
	\begin{split}
	\label{mu.theta.assumptions}
	|\mu(F,\theta)\theta|&<\mu_0|e(F,\theta)|,\\
	|k(F,\theta)|&<k_0|e(F,\theta)|,\\
	\end{split}
	\end{align}
	for some constants $\mu_0,\: k_0>0$,
	we are going to examine how we can obtain measure-valued solutions in the limit as $\mu_0\to0$ and $k_0\to0.$
	We work in a periodic domain in space $Q_T=\mathbb{T}^d\times[0,T),$ for $T\in[0,\infty)$ and $d=3.$
	%For (\ref{poly.sys.mu.k}) to be well-defined, one needs
	%\begin{align*}
	%\Sigma^{\mu,k}_{i\alpha}=\frac{\partial\hat{\psi}}{\partial\xi^B}(\Phi(F^{\mu,k}),\theta^{\mu,k})
	%\frac{\partial\Phi^B}{F_{i\alpha}}(F^{\mu,k})\in L^1(Q_T),\quad
	%\Sigma^{\mu,k}_{i\alpha}v^{\mu,k}_i\in L^1(Q_T),\quad
	%k\nabla\theta^{\mu,k}\in L^1(Q_T),\\
	%\mu\nabla v^{\mu,k}\in L^1(Q_T),\quad 
	%k\frac{\nabla\theta^{\mu,k}}{\theta^{\mu,k}}\in L^1(Q_T),\quad
	%k\frac{|\nabla\theta^{\mu,k}|^2}{(\theta^{\mu,k})^2}\in L^1(Q_T),\quad
	%\mu\frac{|\nabla v^{\mu,k}|^2}{\theta^{\mu,k}}\in L^1(Q_T)\quad
	%\end{align*}
	%and $$\frac{r}{\theta^{\mu,k}}\in L^1(Q_T).$$
	We impose the growth conditions 
	\begin{equation}
	\label{gr.con.1.F}
	c(|F|^p+\theta^{\ell})-c\leq e(F,\theta)\leq c(|F|^p+\theta^{\ell})+c.
	\end{equation} 
	\begin{equation}
	\label{gr.con.2.F}
	c(|F|^p+\theta^{\ell})-c\leq \psi(F,\theta)\leq c(|F|^p+\theta^{\ell})+c\;,
	\end{equation}
	\begin{equation}
	\label{gr.con.3.F}
	\lim_{|F|^p+\theta^{\ell}\to\infty}\frac{|\partial_{\theta}\psi(F,\theta)|}{|F|^p+\theta^{\ell}}=
	\lim_{|F|^p+\theta^{\ell}\to\infty}\frac{|\eta(F,\theta)|}{|F|^p+\theta^{\ell}}=0\;,
	\end{equation}
	and
	\begin{equation}
	\label{gr.con.4.F}
	\lim_{|F|^p+\theta^{\ell}\to\infty}\frac{|\partial_{F}\psi(F,\theta)|}{|F|^p+\theta^{\ell}}=
	\lim_{|F|^p+\theta^{\ell}\to\infty}\frac{|\Sigma(F,\theta)|}{|F|^p+\theta^{\ell}}=0\;,
	\end{equation}
	for some constant $c>0$ and $p,\ell>1.$ 
	
	Integrating the energy equation (\ref{poly.sys.mu.k})$_3$ in $Q_T$ we get
	\begin{align}\label{energynw.unif.bound}
	\int\left(\frac{1}{2}|v^{\mu,k}|^2+\hat{e}(\Phi(F^{\mu,k}),\theta^{\mu,k})\right)\:dx
	\leq 
	%C_{(|\mathbb{T}^d|,T,r)}
	\int\left(\frac{1}{2}|v_0^{\mu,k}|^2+\hat{e}(\Phi(F_0^{\mu,k}),\theta_0^{\mu,k})\right)\:dx%\nonumber\\
	\leq C_1
	\end{align}
	so that
	\begin{align*}
	\sup_{0<t<T}\int\left(\frac{1}{2}|v^{\mu,k}|^2+\hat{e}(\Phi(F^{\mu,k}),\theta^{\mu,k})\right)\:dx\leq C<\infty.
	\end{align*}
	Therefore (\ref{gr.con.1.F}) implies that
	\begin{align}
	\label{norms.unif.bound}
	\int\left(|F^{\mu,k}|^p+\frac{1}{2}|v^{\mu,k}|^2+|\theta^{\mu,k}|^{\ell}\right)\:dx\leq C<\infty
	\end{align}
	and then the functions $(F^{\mu,k},v^{\mu,k},\theta^{\mu,k})$ are all bounded in the spaces:
	\begin{equation}
	\label{Fvtheta.reg}
	F^{\mu,k}\in L^{\infty}(L^p),\quad v^{\mu,k}\in L^{\infty}(L^2), \quad \theta^{\mu,k}\in L^{\infty}(L^{\ell})
	\end{equation}
	and weakly converging to the averages
	\begin{align*}
	%\begin{split}
	%\label{Fvtheta.convergence}
	F^{\mu,k}\rightharpoonup\la \bn,\lf\ra, \quad\text{weak-$\ast$ in\:} L^{\infty} (L^p ) \;,\\
	v^{\mu,k}\rightharpoonup\la \bn,\lv\ra, \quad\text{weak-$\ast$ in\:} L^{\infty}( L^2 ) \;,\\
	\theta^{\mu,k}\rightharpoonup\la \bn,\lt\ra, \quad\text{weak-$\ast$ in\:} L^{\infty}(L^{\ell} ) \, .
	%\end{split}
	\end{align*}
	Integrating now (\ref{poly.sys.mu.k})$_4$ ($r\equiv 0$), in $Q_T$ we obtain
	\begin{align}
	\label{entropy.unif.bound}
	\int\hat{\eta}(\Phi(F^{\mu,k}),\theta^{\mu,k})\:dx-\int\hat{\eta}(\Phi(F^{\mu,k}),\theta^{\mu,k})(x,0)\:dx
	=\int_0^T\!\!\!\!\int\left(k\frac{|\nabla\theta^{\mu,k}|^2}{(\theta^{\mu,k})^2}
	+\mu\frac{|\nabla v^{\mu,k}|^2}{\theta^{\mu,k}}\right)\:dxdt.
	\end{align}
	Then (\ref{gr.con.3.F}) and (\ref{Fvtheta.reg}) immediately imply that 
	$\hat{\eta}(\Phi(F^{\mu,k}),\theta^{\mu,k})\in L^{\infty}(L^1)$ while 
	\begin{align}
	\label{entropy.diss}
	0<\int_0^T\int\left(k\frac{|\nabla\theta^{\mu,k}|^2}{(\theta^{\mu,k})^2}
	+\mu\frac{|\nabla v^{\mu,k}|^2}{\theta^{\mu,k}}\right)\:dxdt\leq C.
	\end{align}
	
	Now, let us consider the first equation in (\ref{poly.sys.mu.k}).
	We employ Lemma \ref{lemma.weak.minors}, in order to pass to the limit in the minors and the identities \eqref{transport.stretching}. It follows that
	(\ref{poly.sys.mu.k})$_1$ holds in the classical weak sense for motions with regularity as in (\ref{y.reg}) and
	\begin{align*}
	(\Phi(F),v,\theta)=(F,\zeta,w,v,\theta)
	\in L^{\infty}(L^p)\times L^{\infty}(L^q)\times L^{\infty}(L^{\rho})\times L^{\infty}(L^2)\times L^{\infty}(L^{\ell})
	\end{align*}
	with $p\ge 4,\:q\ge 2, \rho , \ell>1.$
	
	To pass to the limit in the second equation (\ref{poly.sys.mu.k})$_2,$ 
	we use the Theorem of Ball \cite{MR1036070} on representation via Young measures
	in the $L^p$ setting:
	\begin{lemma} 
		Let $f^{\epsilon}:\bar{Q}_T\to{\mathbb{R}^m}$ be a bounded function in $L^p.$ Then for all $F:\mathbb{R}^m\to\mathbb{R}$ which are continuous and such that $F(f^{\epsilon})$ is $L^1$ weakly precompact,
		there holds (along a subsequence) $$F(f^{\epsilon})\rightharpoonup\la \nu,F\ra, \quad\text{weakly in\:} L^1(Q_T).$$
		
		If $f^{\epsilon}:\bar{Q}_T\to{\mathbb{R}^m}$ is uniformly bounded in $Q_T,$ then for all continuous  $F:\mathbb{R}^m\to\mathbb{R}$ there holds (along a subsequence) $$F(f^{\epsilon})\rightharpoonup\la \nu,F\ra,
		\quad\text{weak-$\ast$ in\:} L^{\infty}(Q_T).$$
	\end{lemma}
	Note that the assumption $\{F(f^{\epsilon})\}$ is $L^1$ weakly precompact cannot be dropped or just replaced by $L^1$
	boundness, because in such case concentrations might develop. We are going to examine each term separately, in order to
	obtain (along a non-relabeled subsequence)
	\begin{align*}
	\partial_t \left\la\bn,\lvi\right\ra&=
	\partial_{\alpha}\left\la\bn,\frac{\partial\hat{\psi}}{\partial \xi^B}(\Phi(\lf),\lt)
	\frac{\partial\Phi^B}{\partial F_{i\alpha}}(\lf)\right\ra
	%\partial_t \left\la\bn,\hat{\eta}(\Phi(\lf),\lt)\right\ra&=
	%\bs_{\mu}+\bs_k+\left\la\bn,\frac{r}{\lt}\right\ra \geq \left\la\bn,\frac{r}{\lt}\right\ra
	\end{align*}
	in the sense of distributions, %for $\bs_{\mu},\bs_k\in\mathcal{M}^+(Q_T),$ 
	which means
	\begin{align*}
	\int\left\la\bn,\lvi\right\ra(x,0)\phi(x,0)\:dx
	&+\int_0^T\int\left\la\bn,\lvi\right\ra \partial_t\phi\:dxdt\\
	&=\int_0^T\int\left\la\bn,\frac{\partial\hat{\psi}}{\partial \xi^B}(\Phi(\lf),\lt)
	\frac{\partial\Phi^B}{\partial F_{i\alpha}}(\lf)\right\ra \partial_{\alpha}\phi\:dxdt
	\end{align*}
	%Let $\varepsilon=(\mu,k)$ and $\fe=U^{\mu,k}.$
	for $\phi(x,t)\in C_c^1(Q_T).$ 
	Observe that (\ref{norms.unif.bound}), (\ref{gr.con.4.F}) combined with (\ref{Sigma}) yield that the term
	$$\frac{\partial\hat{\psi}}{\partial\xi^B}(\Phi(F^{\mu,k}),\theta^{\mu,k})\frac{\partial\Phi^B}{F_{i\alpha}}(F^{\mu,k})$$ is representable with respect to the Young measure $\bn_{(x,t)\in\bar{Q}_T}$ so that
	\begin{align*}
	\frac{\partial\hat{\psi}}{\partial F_{i\alpha}}(\Phi(F^{\mu,k}),\theta^{\mu,k})
	\rightharpoonup\left\la\bn,\frac{\partial\hat{\psi}}{\partial F_{i\alpha}}(\Phi(\lf),\lt)\right\ra, \quad\text{weakly in\:} L^1.
	\end{align*}
	Now, we examine the limit of the diffusion term $\mathrm{div}(\mu\nabla v^{\mu,k})$ as $\mu_0\to 0$. First, we write
	%\begin{align}
	%\begin{split}
	%\label{term1}
	%\left|\int_0^T\int\mathrm{div}(\mu\nabla v^{\mu,k})\phi\:dxdt\right|
	%&=\left|\int_0^T\int(\mu\nabla v^{\mu,k})\cdot\nabla\phi\:dxdt\right|\\
	%&\leq C \mu_0 \|\nabla v^{\mu,k}\|_{L^{\infty}(Q_T)}\int_0^T\int|e(F,\theta)|\:dxdt,
	%\end{split}
	%\end{align}
	\begin{align*}
	\lim_{\mu_0\to 0}\left|\int_0^T\!\!\int\mathrm{div}(\mu\nabla v^{\mu,k})\phi\:dxdt\right|
	&=\lim_{\mu_0\to 0}\int_0^T\!\!\int\left|\mu\nabla v^{\mu,k}\cdot\nabla\phi\right|\:dxdt,
	\end{align*}
	while using Hölder's inequality:
	\begin{align}
	\begin{split}
	\label{term2}
	\lim_{\mu_0\to 0}\int_0^T\!\!\int&\left|\mu\nabla v^{\mu,k}\cdot\nabla\phi\right|\:dxdt\\
	&\leq\lim_{\mu_0\to 0}\left[\left(\int_0^T\!\!\!\int |\mu|\frac{|\nabla v^{\mu,k}|^2}{\theta^{\mu,k}}\:dxdt\right)^{1/2}
	\!\!\!\!\left(\int_0^T\!\!\int |\mu\,\theta^{\mu,k}|\, |\nabla\phi|^2\:dxdt\right)^{1/2}\right]\\
	&\leq\lim_{\mu_0\to 0} C\mu_0=0
	\end{split}
	\end{align}
	for $\phi(x,t)\in C_c^1(Q_T)$, by virtue of (\ref{mu.theta.assumptions})$_1$, (\ref{gr.con.1.F}) and (\ref{entropy.diss}).
	
	Next, we move to the entropy identity (\ref{poly.sys.mu.k})$_4,$ with $r\equiv0.$ In the limit, we aim to have 
	\begin{align*}
	-\int\left\la\bn,\hat{\eta}(\Phi(\lf),\lt)\right\ra(x,0)\phi(x,0)\:dx
	&-\int_0^T\int\left\la\bn,\hat{\eta}(\Phi(\lf),\lt)\right\ra \partial_t\phi\:dxdt\geq 0,
	%	&\geq\int_0^T\int\left\la\bn,\frac{r}{\lt}\right\ra\phi\:dxdt,
	\end{align*}
	for $\phi(x,t)\in C_c^1(Q_T),\:\phi\geq 0.$ Growth condition (\ref{gr.con.3.F}) combined with 
	(\ref{eta&e})$_1$ allows to use again the Fundamental Lemma on Young measures to get
	\begin{align*}	
	\hat{\eta}(\Phi(F^{\mu,k}),\theta^{\mu,k})\rightharpoonup\left\la\bn,\hat{\eta}(\Phi(\lf),\lt)\right\ra, \quad\text{weakly in\:} L^1,
	\end{align*}
	The diffusion term 
	$\mathrm{div}\left(k\frac{\nabla\theta^{\mu,k}}{\theta^{\mu,k}}\right)$ can be treated exactly as in (\ref{term2}),
	using again Hölder's inequality, (\ref{mu.theta.assumptions})$_2$, (\ref{gr.con.1.F}) and (\ref{entropy.diss}):
	\begin{align}
	\begin{split}
	\label{term3}
	\lim_{k_0\to 0}
	&\left|\int_0^T\!\!\!\int\mathrm{div}\left(k\frac{\nabla\theta^{\mu,k}}{\theta^{\mu,k}}\right)\phi\:dxdt\right|\\
	&=\lim_{k_0\to 0}\left|\int_0^T\!\!\!\int k\frac{\nabla\theta^{\mu,k}}{\theta^{\mu,k}}\cdot\nabla\phi\:dxdt\right|\\
	&\leq\lim_{k_0\to 0}\left[\left(\int_0^T\!\!\!\int |k|\,\frac{|\nabla\theta^{\mu,k}|^2}{(\theta^{\mu,k})^2}\:dxdt\right)^{1/2}
	\!\!\!\!\left(\int_0^T\!\!\!\int |k|\,|\nabla\phi|^2\:dxdt\right)^{1/2}\right]\\
	&\leq\lim_{k_0\to 0} Ck_0=0
	\end{split}
	\end{align}
	%
	%\begin{align*}
	%	\int_{\mathbb{T}^d}\left|\sqrt{k}\frac{\nabla\theta^{\mu,k}}{\theta^{\mu,k}}\right|^2\:dx=
	%	\int_{\mathbb{T}^d}k\frac{|\nabla\theta^{\mu,k}|^2}{(\theta^{\mu,k})^2}\:dx
	%	\leq \frac{k_0}{\bar{\delta}}\int_{\mathbb{T}^d}|\theta^{\mu,k}|^2\:dx\leq C
	%\end{align*}
	where $\phi(x,t)\in C_c^1(Q_T),\:\phi\geq 0.$ On account of (\ref{entropy.diss}), we apply the Banach-Alaoglou 
	Theorem to the sequence $\left|\sqrt{k}\frac{\nabla\theta^{\mu,k}}{\theta^{\mu,k}}\right|^2$  to obtain a 
	sequential weak-$\ast$ limit $\bs_k\in\mathcal{M}^+(\bar{Q}_T)$
	\begin{align*}
	\bs_k(\phi)=\int_0^T\int\phi\:d\bs_k
	=\lim_{k_0\to 0}\int_0^T\int\phi\left|\sqrt{k}\frac{\nabla\theta^{\mu,k}}{\theta^{\mu,k}}\right|^2\:dxdt,
	\quad\forall\:\phi\in C(\bar{Q}_T),\:\phi\geq 0.
	\end{align*}
	Similarly, the term
	\begin{align*}
	\mu\frac{|\nabla v^{\mu,k}|^2}{\theta^{\mu,k}}\rightharpoonup \bs_{\mu},
	\quad\text{weak-$\ast$ in $\mathcal{M}^+(\bar{Q}_T),$}
	\end{align*}
	as $\mu_0\to 0,$ because of (\ref{entropy.diss}).
	In summary, in the limit as $k_0,\mu_0\to 0,$ it holds (along a subsequence)
	\begin{align}
	\label{entropy.without.r}
	\partial_t \left\la\bn,\hat{\eta}(\Phi(\lf),\lt)\right\ra=\bs_{\mu}+\bs_k\geq 0
	%\bs_{\mu}+\bs_k+\left\la\bn,\frac{r}{\lt}\right\ra \geq \left\la\bn,\frac{r}{\lt}\right\ra
	\end{align}
	in the sense of distributions.
	
	Next we consider the energy equation (\ref{poly.sys.mu.k})$_3$. 
	The function $$(x,t)\mapsto \big ( \tfrac{1}{2}|v^{\mu,k}|^2 + \hat e (\Phi(F^{\mu,k}), \theta^{\mu,k} ) \big ) dxdt$$ is 
	weakly precompact in the space of nonnegative Radon measures $\mathcal{M}^+(\bar{Q}_T),$ but not 
	weakly precompact in $L^1,$ therefore the Young measure representation fails. To capture the resulting 
	formation of concentrations, we introduce the concentration measure $\bg(dx,dt),$ which is a non-negative Radon 
	measure in $Q_T,$ for a subsequence of $\frac{1}{2}|v^{\mu,k}|^2+\hat{e}(\Phi(F^{\mu,k}),\theta^{\mu,k}),$ according to 
	the analysis in Section 3. As we aim to construct dissipative solutions, in the limit, we consider an integrated
	averaged energy identity, tested against $\varphi=\varphi(t)$ in $C^1_c[0,T],$ that does not depend on the 
	spatial variable; as a result, all the flux terms in~(\ref{poly.sys.mu.k})$_3$ vanish. 
	Having this in mind, the energy equation becomes
	\begin{small}
		\begin{align*}
		\int  \varphi(0) \Big( \la\bn,\frac{1}{2}|\lv|^2+&\hat{e}(\Phi(\lf),\lt)\ra(x,0) \:dx+\bg_0(dx) \Big)\\
		&+\int_0^T \!\!\int \varphi^{\prime}(t)\left(\left\la\bn,\frac{1}{2}|\lv|^2+\hat{e}(\Phi(\lf),\lt)\right\ra\:dx\:dt+\bg(dx\:dt)\right) =0.
		%		\\
		%	    &=-\int_0^T \!\!\!\int \left\la\bn,r\right\ra\varphi\:dx\:dt,
		\end{align*}
	\end{small}
	for all $\varphi=\varphi(t)$ in $C^1_c[0,T]$. Altogether, we conclude:
	\begin{align*}
	\partial_t \Phi^B(F)&=\partial_{\alpha}\left(\phv \right) \\ 
	\partial_t \left\la\bn,\lvi\right\ra
	&=\partial_{\alpha}\left\la\bn,\frac{\partial \hat{\psi}}{\partial\xi^B}(\Phi(\lambda_{F}),\lambda_{\theta})\frac{\partial\Phi^B}{\partial F_{i\alpha}}(\lf)\right\ra\\
	\partial_t \left\la\bn,\hat{\eta}(\Phi(\lambda_{F}),\lt)\right\ra & \geq 0%\left\la\bn,\frac{r}{\lt}\right\ra
	\end{align*}
	and 
	\begin{small}
		\begin{align*}
		\int  \varphi(0) \Big( \la\bn,\frac{1}{2}|\lv|^2+&\hat{e}(\Phi(\lf),\lt)\ra(x,0) \:dx+\bg_0(dx) \Big)\\
		&+\int_0^T \!\!\int \varphi^{\prime}(t)\left(\left\la\bn,\frac{1}{2}|\lv|^2+\hat{e}(\Phi(\lf),\lt)\right\ra\:dx\:dt+\bg(dx\:dt)\right) =0.
		%		\\
		%	    &=-\int_0^T \!\!\!\int \left\la\bn,r\right\ra\varphi\:dx\:dt,
		\end{align*}
	\end{small}
	%\begin{align*}
	%\int&\left(\left\la\bn,\frac{1}{2}|\lv|^2+\hat{e}(\Phi(\lf),\lt)\right\ra\right)(x,0)\varphi(0)\:dx +\bg_0(dx)\\
	%&+\int_0^T
	%\int\left(\left\la\bn,\frac{1}{2}|\lv|^2+\hat{e}(\Phi(\lf),\lt)\right\ra+\bg\right)\varphi^{\prime}(t)\:dx\:dt\\
	%&=-\int_0^T\int \left\la\bn,r\right\ra\varphi(t)\:dx\:dt.
	%\end{align*}
	
	The above analysis indicates the relevance of Definition \ref{def.mv.F} of \emph{dissipative measure valued solution} stated in 
	Section 2.

	\begin{remark}\rm
	Testing the energy equation~(\ref{poly.sys.mu.k})$_3$ against a test function $\varphi=\varphi(x,t)\in C^1_c(Q_T)$, 
	yields a different notion of measure-valued solution in the limit, the so-called \emph{entropy measure-valued} solution.
	In that case,  additional assumptions are required. First, one should represent the term 
	$$\frac{\partial\hat{\psi}}{\partial\xi^B}(\Phi(F^{\mu,k}),\theta^{\mu,k})
	\frac{\partial\Phi^B}{F_{i\alpha}}(F^{\mu,k}) v_i^{\mu,k}$$ in the flux, which requires growth
	conditions on $\Sigma_{i \alpha }(\Phi(F^{\mu,k}),\theta^{\mu,k})v_i^{\mu,k}$. Second, to treat the terms
	$\text{div}(\mu v_i^{\mu,k}\partial_{\alpha}v_i^{\mu,k})$ and  $\text{div}(k\partial_{\alpha}\theta^{\mu,k})$ the additional 
	uniform bounds
	$$
	|\mu(F,\theta) \theta|\leq C\mu_0 \, \quad \mbox{and}   \quad |k(F,\theta) \theta |\leq k_0|e(F ,\theta) | \, ,
	$$
	on the  diffusion coefficients are needed  to pass to the limit. 
	\end{remark}
	\begin{remark} \rm
	Let us return to the notion of \emph{dissipative measure-valued} solution but now with $r\neq0.$ To establish the limit as $\mu_0$, $k_0\to0^+$, we need to represent the terms $r$ and
	$\ds\frac{r}{\theta^{\mu,k}},$ as 
	they appear on the right-hand side of (\ref{poly.sys.mu.k})$_3$ and (\ref{poly.sys.mu.k})$_4$ respectively.
	Then \eqref{energynw.unif.bound} and~(\ref{entropy.unif.bound}) become
	\begin{align*}
	\int\left(\frac{1}{2}|v^{\mu,k}|^2+\hat{e}(\Phi(F^{\mu,k}),\theta^{\mu,k})\right)\:dx
	&\leq \int_0^T\int |{r}|\:dxdt
	%C_{(|\mathbb{T}^d|,T,r)}
	+\int\left(\frac{1}{2}|v_0^{\mu,k}|^2+\hat{e}(\Phi(F_0^{\mu,k}),\theta_0^{\mu,k})\right)\:dx
	\end{align*}
	and
	\begin{align*}
	\int\hat{\eta}(\Phi(F^{\mu,k}),\theta^{\mu,k})\:dx-\int\hat{\eta}(\Phi(F^{\mu,k}),\theta^{\mu,k})(x,0)\:dx
	&=\int_0^T\!\!\!\!\int\left(k\frac{|\nabla\theta^{\mu,k}|^2}{(\theta^{\mu,k})^2}
	+\mu\frac{|\nabla v^{\mu,k}|^2}{\theta^{\mu,k}}\right)\:dxdt\\
	&+\int_0^T\int\frac{r}{\theta^{\mu,k}}\:dxdt
	\end{align*}
	respectively. In turn, to have the bounds~\eqref{norms.unif.bound} and~\eqref{entropy.diss}, the previous analysis to pass to the limit $(\mu_0,k_0)\to(0,0)$ 
	suggests to  require
	\begin{equation}\label{gr.con.1.r}
	r\in L^{\infty}(Q_T)\,\,\text{ and}\quad 0<\bar{\delta}\leq\theta^{\mu,k}
	\end{equation}
	for some $\bar{\delta}>0$, so that both $r$ and $\,\ds\frac{r}{\theta^{\mu,k}}\in L^1(Q_T)$  uniformly in $\mu$ and $k$.
	Hence, as $\mu_0$, $k_0\to0$, we get
	\begin{align*}
	&r\rightharpoonup r\left\la\bn,1\right\ra,\\
	\frac{r}{\theta^{\mu,k}}\rightharpoonup&\left\la\bn,\frac{r}{\lt}\right\ra=r\left\la\bn,\frac{1}{\lt}\right\ra,
	\end{align*}
	weak-$\ast$ in $L^{\infty}(\bar{Q}_T).$
	%Observe that by looking at the integrated energy equation, we do not have to impose any additional integrability 
	%condition on the term  $$v^{\mu,k}\cdot \frac{\partial\hat{\psi}}{\partial F}(\Phi(F^{\mu,k}),\theta^{\mu,k}).$$
	%Also, working with the system of polyconvex thermoelasticity in the variables $(F,v,\theta)$ rather than in the
	%extended variables $(\xi,v,\theta)$ we do not have to impose any growth condition on the derivatives of the free
	%energy function, with respect to the determinant and the cofactor of $F.$
	In summary, the viscosity limit produces a dissipative measure valued solution assuming~\eqref{mu.theta.assumptions} and the growth estimates~\eqref{gr.con.1.F}--\eqref{gr.con.4.F} for $r\equiv 0$ and in addition,~\eqref{gr.con.1.r} for $r\neq 0$.
	\end{remark}

	%\partial_t \Phi^B(F^{\mu,k})&=\partial_{\alpha}\left(\frac{\partial\Phi^B}{\partial F_{i\alpha}}(F^{\mu,k})v_i^{\mu,k}\right)\\  \bs_{\mu},
	%\partial_t v^{\mu,k}_i&=\partial_{\alpha}\left(\frac{\partial\hat{\psi}}{\partial F_{i\alpha}}(\Phi(F^{\mu,k}),\theta^{\mu,k})\right) +\partial_{\alpha}(\mu\partial_{\alpha}v_i^{\mu,k})\\
	%\partial_t \left(\frac{1}{2}|v^{\mu,k}|^2+\hat{e}(\Phi(F^{\mu,k}),\theta^{\mu,k})\right)
	%&=\partial_{\alpha}\left(\frac{\partial \hat{\psi}}{\partial F_{i\alpha}}(\Phi(F^{\mu,k}),\theta^{\mu,k}) v_i^{\mu,k}\right)
	%+\partial_{\alpha}(\mu\partial_{\alpha}v_i^{\mu,k}+k\partial_{\alpha}\theta^{\mu,k})+r\\
	%\partial_t\hat{\eta}(\Phi(F^{\mu,k}),\theta^{\mu,k})&=\partial_{\alpha}\left(k\frac{\nabla\theta^{\mu,k}}{\theta^{\mu,k}}\right)
	%+k\frac{|\nabla\theta^{\mu,k}|^2}{(\theta^{\mu,k})^2}+\mu\frac{|\nabla v^{\mu,k}|^2}{\theta^{\mu,k}}+\frac{r}{\theta^{\mu,k}}

	\section*{Acknowledgments}
	This project has received funding from the European Union's Horizon 2020 programme under the Marie Sklodowska-Curie
		grant agreement No 642768. Christoforou was partially supported by the Internal grant SBLawsMechGeom \#21036 from University of Cyprus.

\end{document}